\documentclass[12pt]{article}

\usepackage{graphicx}
\usepackage{amsfonts}
\usepackage{multicol}
\usepackage[utf8]{inputenc}
\usepackage{graphicx}
\usepackage{caption}
\usepackage{setspace}
\usepackage{float}
\usepackage{lettrine}
\usepackage{changepage}
\usepackage[english]{babel}
\usepackage{gensymb}
\usepackage{amsmath}
\usepackage{polynom}
\usepackage{bm}
\usepackage{subfig}
\usepackage{booktabs}
\usepackage{hyperref}
\usepackage{enumerate}
\usepackage{polynom}
\usepackage{enumitem}
\usepackage[margin=1.0in]{geometry}
\usepackage{mathtools}
\usepackage{algorithm}
\usepackage{algpseudocode}
\usepackage[mathscr]{euscript}
\usepackage{tikz}
\usepackage{stmaryrd}
\usetikzlibrary{positioning}
\usepackage[page]{appendix}
\usepackage{rotating}
\usepackage{lscape}
\usepackage{amsthm}
\usepackage{placeins} 
\usepackage{amssymb}
\usepackage[group-separator={,}]{siunitx}

\mathtoolsset{showonlyrefs}

\algnewcommand\algorithmicinput{\textbf{Input:}}
\algnewcommand\INPUT{\item[\algorithmicinput]}

\allowdisplaybreaks

\newcommand{\R}{\mathbb{R}}
\newcommand{\X}{\mathcal{X}}
\renewcommand{\S}{\mathcal{S}}
\newcommand{\Y}{\mathcal{Y}}
\newcommand{\W}{\mathcal{W}}
\newcommand{\F}{\mathcal{F}}
\newcommand{\T}{\mathcal{T}}
\newcommand{\D}{\mathcal{D}}
\newcommand{\Z}{\mathcal{Z}}
\newcommand{\LC}{\mathcal{L}}

\newcommand{\PP}{\mathcal{P}}
\newcommand{\sgn}{\textnormal{sgn}}
\newcommand{\<}{\langle}
\renewcommand{\>}{\rangle}
\newcommand{\nsum}{\textnormal{sum}}
\newcommand{\supp}{\textnormal{supp}}
\newcommand{\Rc}{\mathcal{R}}
\newcommand{\rank}{\textnormal{rank}}
\newcommand{\vertiii}[1]{{\left\vert\kern-0.25ex\left\vert\kern-0.25ex\left\vert #1 
    \right\vert\kern-0.25ex\right\vert\kern-0.25ex\right\vert}}

\newcommand{\lambdaX}{\lambda_{x}} 
\newcommand{\lambdaS}{\lambda_{s}}
\newcommand{\coherence}{\mu}
\newcommand{\Q}{\mathcal{Q}}

\newtheorem{theorem}{Theorem}
\newtheorem{lemma}[theorem]{Lemma}
\newtheorem{corollary}[theorem]{Corollary}
\newtheorem{proposition}[theorem]{Proposition}

\newcommand{\defeq}{\stackrel{\text{\tiny def}}{=}}

\DeclareMathOperator*{\argmin}{arg\,min}

\begin{document}

\title{Tensor Robust Principal Component Analysis: Better recovery with atomic norm regularization}

\author{%
  Derek Driggs\footnote{Applied Mathematics and Theoretical Physics, University of Cambridge (\texttt{d.driggs@cam.ac.uk})}%
  \and Stephen Becker\footnote{Applied Mathematics, University of Colorado Boulder (\texttt{stephen.becker@colorado.edu})}%
  \and Jordan Boyd-Graber\footnote{Computer Science, iSchool, lsc, and umiacs, University of Maryland College Park (\texttt{jbg@umiacs.umd.edu})}%
  }
\date{}

\maketitle

\begin{abstract}
This paper studies tensor-based Robust Principal Component Analysis (RPCA) using atomic-norm regularization. Given the superposition of a sparse and a low-rank tensor, we present conditions under which it is possible to exactly recover the sparse and low-rank components. Our results improve on existing performance guarantees for tensor-RPCA, including those for matrix RPCA. Our guarantees also show that atomic-norm regularization provides better recovery for tensor-structured data sets than other approaches based on matricization.

In addition to these performance guarantees, we study a nonconvex formulation of the tensor atomic-norm and identify a class of local minima of this nonconvex program that are globally optimal. We demonstrate the strong performance of our approach in numerical experiments, where we show that our nonconvex model reliably recovers tensors with ranks larger than all of their side lengths, significantly outperforming other algorithms that require matricization.
\end{abstract}

\vspace{3mm}


\vspace{3mm}


\section{Introduction}

Tensors, or multidimensional arrays, are becoming increasingly prominent in data analysis and machine learning. Tensors were first used as tools for data analysis in the psychometics community, where researchers used tensor decompositions to study fMRI data sets that are more naturally represented as tensors than matrices \cite{KoldaRev,Psycho}. Since then, tensors have established a place in chemometrics, computer vision, compressed sensing, data mining, and higher-order statistics \cite{SidiReview2017}. Tensors' ability to naturally represent the distributions of latent variable models has also made them important tools for learning a variety of latent variable models, including topic models, Gaussian mixture models, and neural networks, to name a few \cite{AnandHsu,NNLift}. With modern data sets growing quickly in both size and complexity, tensor-based algorithms offer more natural approaches for analyzing multidimensional data.

Many tensor-based algorithms for data analysis are formulated as low-rank recovery problems. Low-rank tensor decompositions have been used for video processing \cite{AnandRPCA,SobralTensor}, topic modeling \cite{Anand1}, blind source separation \cite{ComonBSS}, and parameter estimation in latent variable models \cite{AnandHsu}. Low-rank matrix recovery problems have been studied extensively, and Robust Principal Component Analysis (RPCA) is one approach to this problem \cite{RPCA}. RPCA decomposes the superposition of a low-rank and a sparse matrix into their original components by solving a convex optimization program, thereby efficiently recovering a low-rank matrix from grossly corrupted measurements.

Recently, several authors have proposed tensor completion algorithms for low-rank recovery \cite{BarakMoitra,GoldfarbTensors,SobralTensor,YuanZhang,MomentTensorAtom,ExactTensorSumSquares}, and only a few of these techniques have been extended to RPCA \cite{AnandRPCA,SNN,TangTensorRPCA,TubalRPCA}. Many of these algorithms do not work with tensors directly, but instead work with matrix representations of the higher-order data. Recent work has shown that representing tensors as matrices leads to sub-optimal performance \cite{YuanZhang}, but there are few methods for low-rank tensor recovery that are tractable, have performance guarantees, and preserve the tensor's higher-order structure. 

In this paper, we study the following model for tensor RPCA:
\begin{equation}
\min_{\X,\S} \quad \| \X \|_* + \lambda \| \S \|_{\textnormal{sum}} \quad \textnormal{subject to:} \quad \X + \S = \Z
\end{equation}
Our model uses the tensor atomic-norm: a higher-order generalization of the matrix nuclear-norm (see Section \ref{sec:Prelim1} for further details). We present the first performance guarantees for tensor RPCA using the tensor atomic norm. Our results improve on all existing recovery guarantees for tensor RPCA, corroborating recent literature suggesting that preserving the structure of multidimensional data sets allows for significantly improved recovery. The order-two case of our bounds also offer slightly improved recovery guarantees for matrix RPCA. We also discuss a nonconvex algorithm for solving our tensor RPCA program, and show that although our problem is nonconvex, all local minima are globally optimal. Our experiments show that our model can consistently recover tensors with full Tucker-rank (but not full CP-rank). This suggests that although our performance guarantees are sharp using the Tucker-rank as a metric, they are not yet sharp in terms of the CP-rank.

The rest of this paper is outlined as follows. In Section \ref{sec:Prelim1}, we provide an overview of some analytic and algebraic properties of tensors. In Section \ref{sec:MainRes}, we present recovery guarantees for our formulation of tensor RPCA using atomic-norm regularization. We also compare these guarantees to existing results for tensor RPCA. Section \ref{sec:lem:dualCert} contains the proof of these recovery guarantees.
While the tensor atomic norm is convex, it is in general computationally intractable, so in Section \ref{sec:NonCon}, we discuss a nonconvex representation of the tensor atomic norm that can be seen as a higher-order generalization of the Burer-Monteiro factorization approach that is popular in low-rank matrix recovery algorithms. This formulation has been used previously for tensor completion, and we show how it can be extended to tractably find stationary points of the tensor RPCA program. We also show that all local minima of our nonconvex program are globally optimal. We discuss how tensor RPCA cam be used for training latent variable models in Section \ref{sec:topmod}, using topic modeling as a motivating example. Finally, we present numerical experiments in Section \ref{numerics} that demonstrate the efficacy of our approach.

\section{Tensor Preliminaries}
\label{sec:Prelim1}

This section introduces definitions and properties of tensors. A more complete review of this information can be found in \cite{Friedland}, \cite{KoldaRev}, and \cite{YuanZhang}. We focus on order-three tensors to simplify notation, but our results can easily be extended to arbitrary orders.

Let $\X \in \R^{d_1 \times d_2 \times d_3}$ be an order-three tensor with side lengths $d_1,d_2,d_3$. The \textit{fibers} of $\X$ \textit{along its} $k^{th}$ \textit{mode} are the vectors obtained by holding all but one of the indices of $\X$ fixed and varying the $k^{th}$ index. 
In some cases, it is useful to \textit{matricize} a tensor, so that it is represented as a matrix. The matricization of an order-three tensor $\X$ along its $k^{th}$ mode is denoted $\X_{(k)}$ and is formed by taking the mode-$k$ fibers of $\X$ and making them the columns of $\X_{(k)}$.

The tensor product, which we denote $\otimes$, is defined so that if $\X = u \otimes v \otimes w$, then $\X_{i_1,i_2,i_3} = u_{i_1} v_{i_2} w_{i_3}$. The tensor product generalizes the outer product, so $u \otimes v = uv^T$.
Every tensor $\X$ admits a \textit{CP-decomposition} (CPD) of the form
\begin{align}
\label{cp}
\X &= \sum_{r=1}^R (a_r \otimes b_r \otimes c_r) = \sum_{r=1}^R \gamma_r (u_r \otimes v_r \otimes w_r),
\end{align}
where the vectors $u_r,v_r$, and $w_r$ have unit-norm and $\gamma_r = \|a_r\| \|b_r\| \|c_r\|$. When $R$ is minimal, we call $R$ the \textit{CP-rank} of $\X$. The matrices $A,B,C$ that have $a_r, b_r, c_r$ as their columns, respectively, are \textit{factor matrices} of $\X$. It is also sometimes convenient to use Kruskal's notation $\X = \llbracket A,B,C\rrbracket=\llbracket \gamma;U,V,W\rrbracket$ to denote the decomposition in \eqref{cp}.

Matrices can act on a tensor through multiplication. For a tensor $\X = \llbracket A,B,C\rrbracket \in \R^{d_1 \times d_2 \times d_3}$, multiplication by the matrices $M_i \in \R^{k_i \times d_i}, \ i = 1,2,3$ is defined as follows:
\begin{equation}
(M_1,M_2,M_3) \cdot \X = \llbracket M_{1} A, M_{2} B, M_{3} C \rrbracket \in \R^{k_1 \times k_2 \times k_3},
\end{equation}
and multiplication between the factor matrices is canonical matrix multiplication. We choose to use the notation of \cite{LimRank}, but this is multiplication is the same as $\X \times_1 M_1 \times_2 M_2 \times_3 M_3$ using the notation of \cite{KoldaRev}.

We also use the \textit{Khatri-Rao} product, which we denote $\odot$. For two matrices $A \in \R^{m \times n}. \ B \in \R^{p \times n}$ with the same number of columns, we have
\begin{equation}
\label{eq:khatrirao}
A \odot B \defeq \left( \begin{array}{c c c}
a_{1,1} \mathbf{b_1} & \cdots & a_{1,n} \mathbf{b_n} \\
\vdots & \ddots & \vdots \\
a_{m,1} \mathbf{b_1} & \cdots & a_{m,n} \mathbf{b_n}
\end{array} \right) \in \R^{m \cdot p \times n},
\end{equation}
where $\mathbf{b_i}$ is the $i^{th}$ column of $B$. A matricized tensor can be expressed neatly using the Khatri-Rao product \cite{KoldaRev}.

The \textit{Tucker-rank} of $\X$ is the tuple $\left(\rank(\X_{(1)}), \rank(\X_{(2)}), \rank(\X_{(3)})\right)$. We can bound the CP-rank of a third-order tensor using a weighted average of the components of its Tucker-rank:
\begin{equation}
\overline{r}(\X) \defeq \sqrt{\frac{r_1 r_2 d_3 + r_1 r_3 d_2 + r_2 r_3 d_1}{d_1 + d_2 + d_3}},\quad \text{where}\;r_i \defeq \rank(\X_{(i)}).
\end{equation}
If $d_1=d_2=d_3=d$ then $\overline{r}\le d$.
It has been established that the CP-rank $R \in [\overline{r},\overline{r}^2]$ \cite{YuanZhang}.

\subsection{Coherence}\label{sec:coherence}

Exact recovery of a tensor through our RPCA model relies on the tensor having low coherence. We adopt the measures of tensor coherence introduced in \cite{YuanZhang}. Recall that the coherence of an $r$-dimensional linear subspace $\text{span}(U)$ of $\mathbb{R}^k$ is defined to be \cite{CandesRecht2009,YuanZhang}
\begin{equation*}
\mu(U) \defeq \frac{k}{r} \max_{1 \le i \le k} \|\PP_U e_i\|^2 = \frac{\max_{1 \le i \le k} \|\PP_U e_i \|^2}{k^{-1} \sum_{i=1}^k \|\PP_U e_i\|^2},
\end{equation*}
where $\PP_U$ is the projection onto $\text{span}(U)$. For a tensor $\X = \llbracket A,B,C\rrbracket \in \mathbb{R}^{d_1 \times d_2 \times d_3}$, we define one measure of coherence as
\begin{equation} \label{eq:coherence}
\coherence(\X) \defeq \max\{ \mu(A), \mu(B), \mu(C) \}.
\end{equation}
We can also interpret $\coherence(\X)$ as the maximum coherence of the column spaces of each matricization of $\X$. Another measure of coherence, for $\X=\llbracket A,B,C \rrbracket$, is
\begin{equation}
\label{eq:alphaDef}
\alpha(\X) \defeq \sqrt{d_1d_2d_3/\overline{r}(\X)} \|\W\|_{\max},
\end{equation}
where $\W = \llbracket W_1,W_2,W_3\rrbracket$ satisfies $\W = \llbracket \PP_A W_1, \PP_B W_2, \PP_C W_3\rrbracket$, $\|\W\|=1$,  $\<\X,\W\> = \|\X\|_*$, and
$\|\W\|_{\max}$ is the largest entry of $\W$ in absolute value (cf.~\eqref{eq:dualTensor}).

We assume that the low-rank component $\X \in \mathbb{R}^{d_1 \times d_2 \times d_3}$ has low coherence, so it satisfies $\coherence(\X) \le \coherence_0$ and $\alpha(\X) \le \alpha_0$. These coherence bounds are not especially more restrictive than the coherence bounds required for low-rank matrix recovery. A bounded coherence ensures that the low-rank component of $\X$ is not sparse, so it can be separated from the sparse component. $\coherence_0$ bounds the coherence of each matricization of $\X$, and $\alpha_0$ provides a uniform bound on these coherences similar to the uniform bound used in \cite{SNN}.

    \subsection{Projection Operators}
    
    In the proof of our main result, we use projection operators that act on order-three tensors. Let $U,V,W$ be matrices. With $\X = \llbracket A,B,C\rrbracket$, define the projection operator $\PP_{U,V,W} : \X \mapsto \llbracket \PP_U(A),\PP_V(B),\PP_W(C)\rrbracket$, where $\PP_U$, for example, projects matrices onto the column space of $U$,
    and $\PP_{U^\perp}$ projects onto the orthogonal complement of the column space of $U$. For convenience, we adopt the notation of \cite{YuanZhang} to define the following projections:
    \begin{align*}
    \PP_{\X}^0 &\defeq \PP_{U,V,W} \\
    \PP_{\X} &\defeq \PP_{U,V,W} + \PP_{U^{\perp},V,W} + \PP_{U,V^{\perp},W} + \PP_{U,V,W^{\perp}}\\
    \PP_{\X_1} &\defeq \PP_{U^{\perp}, V^{\perp}, W} \\
    \PP_{\X_2} &\defeq \PP_{U^{\perp}, V, W^{\perp}} \\
    \PP_{\X_3} &\defeq \PP_{U, V^{\perp}, W^{\perp}} \\
    \PP_{\X_4} &\defeq \PP_{U^{\perp}, V^{\perp}, W^{\perp}} \\
    \PP_{\X^{\perp}} &\defeq \PP_{\X_1} + \PP_{\X_2} + \PP_{\X_3} + \PP_{\X_4} = I-\PP_{\X}. \\
    \end{align*}
These are orthogonal projections, i.e., if $\PP$ is any of the above, then $\PP^2=\PP$ and $\PP=\PP^\top$.

Let $\S$ be a sparse tensor. The \textit{support} of $\S$, which we denote $\Omega$, is the set of indices corresponding to the nonzero entries of $\S$. We use $\PP_{\Omega}$ to denote the projection onto the support of $\S$, i.e.,
\begin{equation}
\left(\PP_{\Omega}(\X) \right)_{i_1,i_2,i_3} = \begin{cases}
\X_{i_1,i_2,i_3} & (i_1,i_2,i_3) \in \Omega, \\
0 & (i_1,i_2,i_3) \not \in \Omega.
\end{cases}
\end{equation}

    \subsection{Norms for Tensors and Operators on Tensors}\label{sec:norms}

    In the proof of our main result, we analyze linear operators acting on order-three tensors and their operator norms, which we define with respect to the usual Euclidean inner product. Let
    $$ \<\X,\Y \> \defeq \sum_{(i,j,k) \in [d_1] \times [d_2] \times [d_3]} 
    \X_{i,j,k} \Y_{i,j,k}.
    $$
    The Frobenius (or Hilbert-Schmidt) norm is the induced norm
    $ \|\X\|_F^2 = \<\X,\X\>$. 
    Let $\mathcal{Q} : \mathbb{R}^{d_1 \times d_2 \times d_3} \to \mathbb{R}^{d'_1 \times d'_2 \times d'_3}$ be a
    linear operator.
    We define the norm of $\mathcal{Q}$ as the operator norm 
    \begin{equation}
    \label{opNorm}
    \vertiii{\mathcal{Q}} \defeq \sup_{\substack{\X \in \mathbb{R}^{d_1 \times d_2 \times d_3} \\ \|\X\|_F \le 1 } } \|\mathcal{Q}(\X)\|_F.
    \end{equation}

    Finally, we need higher-dimensional generalizations of the matrix $\ell_1$ and $\ell_{\infty}$ norms.
    Write $[d_1]$ as shorthand for the list $(1,2,\ldots,d_1)$. Then for a tensor $\X \in \R^{d_1 \times d_2 \times d_3}$, we can define
    \begin{equation} \label{eq:sumNorm}
    \|\X\|_{\textnormal{sum}} \defeq \sum_{(i,j,k) \in [d_1] \times [d_2] \times [d_3]} |\X_{i,j,k}|,
    \end{equation}
    and
    \begin{equation}\label{eq:maxNorm}
    \|\X\|_{\max} \defeq \max_{(i,j,k) \in [d_1] \times [d_2] \times [d_3]} |\X_{i,j,k}|.
    \end{equation}

The \emph{subdifferential} of a function $f$ at a point $x$ is the set $\partial f (x) \defeq \{ d \mid f(y) \ge f(x) + \<d,y-x\> \,\forall y \in \text{dom}(f)\}$. The subdifferential of $\X\mapsto \|\X\|_{\textnormal{sum}}$ reduces to
$$
\partial \|\X\|_{\textnormal{sum}} = \{ \sgn(\X) + \F \mid \F = \PP_{\Omega^\perp}\F, \,  \|\F\|_{\max}\le 1 \} \quad\left(\Omega = \supp(\X) \right)
$$
where $\sgn(\X)$ computes the sign of $\X$ element-wise, with $\sgn(0)=0$. Note $\F = \PP_{\Omega^\perp}\F \iff \PP_\Omega\F=0$.

The tensor \textit{atomic norm} (also known as the tensor \textit{nuclear norm}), which we denote as $\| \cdot \|_*$, is defined as follows:
\[ \| \X \|_* = \min \left\{ \sum_{r=1}^R | \gamma_r | : \X = \sum_{r=1}^R \gamma_r (u_r \otimes v_r \otimes w_r), \ \| u_r \| = \| v_r \| = \| w_r \| = 1 \right\} \]
In the matrix case, the atomic norm is the matrix nuclear norm, which is equal to the sum of the singular values of a matrix. 
Unlike for matrices, the decomposition that realizes the minimum above
is not necessarily the minimal-rank CPD of $\X$. We call a decomposition that does achieve the minimum an \textit{atomic decomposition} of $\X$, and such a decomposition always exists \cite{Friedland}.
The corresponding number of terms $R$ in an atomic decomposition is \emph{atomic rank}, as we use in Section \ref{subsec:burmontfact}. 
The tensor atomic norm can be interpreted as the $\ell_1$-norm of the weights in its atomic decomposition. Roughly speaking, atomic-norm regularization encourages these weights to tend toward zero, promoting low-rank solutions just as $\ell_1$-norm regularization promotes sparsity.

    The dual to the atomic norm is the spectral norm, $\| \cdot \|$, which is defined as \cite{Friedland,YuanZhang}:
    \begin{equation}
    \label{eq:w}
    \| \X \| =
    \max_{ \| u \| = \| v \| = \| w \| = 1 } \< \X, u \otimes v \otimes w \>
    \end{equation}
      There always exists a unit-Euclidean-norm tensor that maximizes this inner-product, so the maximum is well-defined \cite{Friedland}. Furthermore, there always exists a \textit{dual tensor}~\cite{YuanZhang} where
\begin{equation}
\label{eq:dualTensor}
\W\;\text{is dual to}\; \X \iff  \left(   \| \W \| = 1,\; \PP_{\X}^0 \W = \W\; \text{and}\; \langle \X, \W \rangle = \| \X \|_* \right)
\end{equation}
      The spectral and atomic norms satisfy $\| \X \| \le \| \X \|_F \le \| \X \|_*$. To see this, let $\X = \sum_{r=1}^R \gamma_r (u_r \otimes v_r \otimes w_r)$ be an atomic decomposition of $\X$. By the triangle inequality,
      \begin{equation}
      \label{eq:fronuc}
      \| \X \|_F \le \sum_{r=1}^R \gamma_r \|u_r \otimes v_r \otimes w_r\|_F = \sum_{r=1}^R \gamma_r = \| \X \|_*.
      \end{equation}
The inequality $\| \X \| \le \|\X\|_F$ follows from the fact that the spectral norm is dual to the atomic norm.
      
      The proof of our main result relies on the following partial characterization of the subdifferential of the tensor atomic norm, which can be found in \cite{YuanZhang}:
\begin{equation} \label{eq:subdiff}
\partial \| \X \|_* \supset \{ \W + \PP_{\X^{\perp}} \W^{\perp} : \| \W^{\perp} \| \le \tfrac{1}{2}, \ \W \textnormal{ is dual to } \X \}.
\end{equation}

\section{Main Result}
\label{sec:MainRes}

Given the sum of a low-rank tensor $\X$ and a sparse tensor $\S$, $\Z = \X + \S$, we would like to recover $\X$ and $\S$ by solving the following program:
\begin{equation}
\label{RPCA}
\min_{\X',\S'} \quad \| \X' \|_* + \lambda \| \S' \|_{\textnormal{sum}} \quad \textnormal{subject to:} \quad \X' + \S' = \Z
\end{equation}
Theorem \ref{main} provides conditions under which \eqref{RPCA} recovers $\X$ and $\S$ exactly with high probability.
\begin{theorem}
\label{main}
Suppose tensor $\X \in \R^{d_1 \times d_2 \times d_3}$ satisfies $\coherence(\X) \le \coherence_0$ and $\alpha(\X)$ $\le \alpha_0$. Let $\S \in \R^{d_1 \times d_2 \times d_3}$ have a support set $\Omega$ that is uniformly distributed among all sets of cardinality $m$, and let $n = d_1 d_2 d_3 - m$. There then exists a positive constant $c$ so that \eqref{RPCA} with $\lambda = (d_1+d_2+d_3)^{-1/2}$ exactly recovers $\X$ and $\S$ with probability $1-(d_1+d_2+d_3)^{-1-c}$, provided that
\begin{equation}
\label{eq:mainResult}
\overline{r}(\X) \le \rho_r \left( \frac{n}{(d_1+d_2+d_3) \log(n) \alpha_0^4 \coherence_0^2 } \right)^{1/2} \quad \textnormal{\textit{and}} \quad m \le \rho_s d_1 d_2 d_3.
\end{equation}
where $\rho_r, \rho_s$ are numerical constants.
\end{theorem}
Our guarantees improve on existing bounds for tensor RPCA, and they are also an improvement over performing matrix RPCA on the matricized tensor \cite{RPCA}. 

Furthermore, these bounds are near optimal in the following sense. For simplicity, consider these bounds in the regime $d_1 = d_2 = d_3 = d$, $r_1 = r_2 = r_3 = r$ and $m \ll d^3$. The rank-bound in equation \eqref{eq:mainResult} is then
\begin{equation}
\label{eq:ourbounds}
r \le \rho_r \left( \frac{d}{ \log(d)^{-1/2} \alpha_0^2 \coherence_0 } \right).
\end{equation}
The maximum possible Tucker-rank is $r_1 = r_2 = r_3 = d$, so our program allows for exact recovery of a tensor with full Tucker-rank to within a factor of $\log(d)^{-3/2}$. In fact, our numerical experiments in Section \ref{numerics} indicate that our nonconvex reformulation of \eqref{RPCA} can often recover tensors of full Tucker-rank (but low CP-rank). These guarantees, along with our numerical experiments, suggest that the Tucker-rank is often a poor measure of the complexity of tensors when compared to the CP-rank.

Outside of this work, the best guarantees (in many regimes) for a tensor RPCA are from \cite{SNN}. The relevant theorem from this paper is listed below for the case $K=3$ and $d_1=d_2=d_3=d$:
\begin{theorem}[Thm. 1 in \cite{SNN}]
\label{thm:SNN}
Consider the following sum-of-nuclear-norms model for tensor RPCA:
\begin{equation}
\label{eqSNN}
\min_{\X', S'} \quad d \sum_{i = 1}^3  \|\X'_{(i)}\|_* + \|\S'\|_1, \quad \textnormal{subject to:} \quad \X' + \S' = \Z,
\end{equation}
Let $\X$ have Tucker-rank $(r_1,r_2,r_3)$ and let $S$ have support set $\Omega$ that is uniformly distributed among all sets of cardinality $m$. Then there exists a constant $C$ such that
\eqref{eqSNN} exactly recovers $\X$ and $\S$ with probability 
$1-C d^{-3}$
as long as
\begin{equation}
\label{eq:SNNbounds}
r_k \le C_r  \frac{\iota_0^{-1} d}{\log^2d}, \quad \textnormal{\textit{and}} \quad m \le \rho d^3, \quad(\forall\; k=1,2,3),
\end{equation}
for some constants $C_r$ and $\rho$, and incoherence parameter $\iota_0$.
\end{theorem}
Comparing \eqref{eq:ourbounds} with \eqref{eq:SNNbounds}, we see that our bound is less restrictive on the rank of the tensor by a logarithmic factor. The incoherence parameter $\iota_0$ used in \cite{SNN} satisfies $\mu_0 \le \iota_0$, so this improvement is valid no matter the coherence of $\X$. More importantly, Theorem \ref{thm:SNN} bounds the maximum component of the Tucker-rank, while our result bounds a weighted average of all the components, so it is less restrictive in this sense as well.

While the guarantees of Theorem \ref{thm:SNN} are the best existing bounds in many regimes, there are other results that are stronger in special cases. \cite{AnandRPCA} develop a nonconvex approach to tensor RPCA with associated convergence and performance guarantees. Their results are stronger than those in Theorem \ref{thm:SNN} when $\S$ is block-sparse. Because our model assumes that the support of $\S$ is uniformly distributed, direct comparisons with \cite{AnandRPCA} are difficult. \cite{TubalRPCA} offer another set of guarantees for tensor RPCA based on the tubal rank of a tensor (see also \cite{NewTubalRPCA}). Because the results in \cite{TubalRPCA} set bounds on the tubal rank of the low-rank component, it is difficult to directly compare these bounds to others, but these results are similar to those in Theorem \ref{thm:SNN}. There are also some works that present non-convex algorithms for decomposing a tensor into low-rank and sparse components in the presence of noise \cite{cubicSketch,store}. While the recovery guarantees of these papers are similar to ours, they are not directly applicable to the noiseless regime.

    \section{Proof of Theorem \ref{main}}
    \label{sec:lem:dualCert}
    
    Our proof of Theorem \ref{main} is similar to the proof for the matrix case \cite{RPCA}. We use a partial characterization of the subdifferential of the tensor atomic norm to formulate a dual certificate that ensures exact recovery. We then prove that such a certificate exists with high probability provided that the rank of the low-rank component and the sparsity of the sparse component satisfy the bounds given in Theorem \ref{main}.
    
    Lemma \ref{lem:dualCert} establishes our dual certificate. It relies on the condition that $\vertiii{ \PP_{\Omega} \PP_{\X} } < \frac{1}{2}$; we discuss conditions under which $\vertiii{ \PP_{\Omega} \PP_{\X} }$ is bounded with high probability in Section \ref{sec: boundingTheOp}.

    \begin{lemma}
        \label{lem:dualCert}
        Suppose $\vertiii{ \PP_{\Omega} \PP_{\X} } < \frac{1}{2}$ and $\lambda \in (0, 1)$. Then $(\X,\S)$ is the unique  solution to \eqref{RPCA} if $\Z = \X + \S$ and 
        there exists tensors 
        $\W^{\perp}, \D$ and $\F$ satisfying
        \begin{equation}
        \label{eq:mainLemma}
        \W + \PP_{\X^{\perp}} \W^{\perp} = \lambda (\sgn(\S) + \F + \PP_{\Omega} \D)
        \end{equation}
        where $\| \W^{\perp} \| \le \frac{1}{4}$, $\PP_{\Omega} \F = 0$, $\| \F \|_{\max} \le \frac{1}{4}$, $\| \PP_{\Omega} \D \|_F \le \frac{1}{8}$,
        and $\Omega = \supp(\S)$. (As before and throughout this paper, $\W$ satisfies $\PP_{\X}^0 \W = \W$, $\|\W\| = 1$, and $\big< \W,\X \big> = \|\X\|_*$, i.e. $\W$ is dual to $\X$.)
    \end{lemma}

    \begin{proof}
Let $\Delta$ be a perturbation away from the supposed optimal point $\X$,
so $(\X + \Delta, \S - \Delta)$ is a feasible point of \eqref{RPCA}.        
Let $\W$ be dual to $\X$ and $\|\W^{\perp}\|\le \frac{1}{2}$, so $\W + \PP_{\X^{\perp}} \W^{\perp} \in \partial \|\X\|_*$, and let $\sgn(\S) + \F \in \partial \|\S\|_\text{sum}$.
 
We then have
        \begin{align*}
        \| \X + \Delta \|_* + \lambda \| \S - \Delta \|_{\textnormal{sum}} \ge \| \X \|_* + &\< \Delta, \W + \PP_{\X^{\perp}} \W^{\perp} \big> \\
        & + \lambda \left( \|\S\|_{\textnormal{sum}} - \<\Delta, \sgn(\S) + \F \big>  \right)
        \end{align*}
        By the duality of the tensor atomic and spectral norms, there exists a 
        $\W^{\perp}$ satisfying $\|\W^{\perp}\| = \frac{1}{2}$ and $\big< \PP_{\X^{\perp}} \Delta, \W^{\perp} \big> = \frac{1}{2} \| \PP_{\X^{\perp}} \Delta \|_*$. Similarly, we can choose $\F$ so that $\big< \F, \Delta \big> = -\frac{1}{2} \| \PP_{\Omega^{\perp}} \Delta \|_{\textnormal{sum}}$. Hence,
        \begin{align}
        & \|\X + \Delta\|_* - \| \X \|_* +  \lambda (\|\S - \Delta\|_{\nsum} - \| \S \|_{\nsum}) \\
        & \quad \ge \frac{1}{2} \left( \|\PP_{\X^{\perp}} \Delta \|_* + \lambda \| \PP_{\Omega^{\perp}} \Delta \|_{\nsum} \right) + \<\Delta, \W - \lambda \sgn(\S) \>
        \end{align}
        We would like to show that the right side of this inequality is positive unless $\Delta = 0$. To this end, we bound the magnitude of the last term. Using equation \eqref{eq:mainLemma},
        \begin{align*}
        \big|\<\Delta, \W - \lambda \sgn(\S) \big>\big| &= |\big< -\PP_{\X^{\perp}} \W^{\perp} + \lambda \F + \lambda \PP_{\Omega} \D, \Delta \big>| \\
        &\le |\big<\PP_{\X^{\perp}} \W^{\perp}, \Delta \big>| + \lambda |\big< \F, \Delta \big>| + \lambda |\big< \PP_{\Omega} \D, \Delta \big>| \\
        &\le \|\W^{\perp}\| \|\PP_{\X^{\perp}} \Delta\|_* + \lambda \|\F\|_{\max} \|\PP_{\Omega^\perp} \Delta\|_{\nsum} \\
        & \quad \quad + \lambda \| \PP_{\Omega} \D \|_F \| \PP_{\Omega} \Delta \|_F \\
        &< \frac{1}{4} \left( \|\PP_{\X^{ \perp}} \Delta \|_* + \lambda \| \PP_{\Omega^{\perp}} \Delta \|_{\nsum} \right) + \frac{\lambda}{8} \| \PP_{\Omega} \Delta \|_F,
        \end{align*}
        where we used the facts that $\PP_{\X^{\perp}}$ and $\PP_{\Omega}$ are self-adjoint and $\F$ is supported on $\Omega^{\perp}$. This yields
        \begin{align*}
        & \|\X + \Delta\|_* - \| \X \|_* +  \lambda (\|\S - \Delta\|_{\nsum} - \| \S \|_{\nsum}) \\
        & \quad > \frac{1}{4} \left( \|\PP_{\X^{\perp}} \Delta \|_* + \lambda \| \PP_{\Omega^{\perp}} \Delta \|_{\nsum} \right) - \frac{\lambda}{8} \| \PP_{\Omega} \Delta \|_F
        \end{align*}
        The last term can be bounded.
        \begin{align*}
        \|\PP_{\Omega} \Delta \|_F &= \|\PP_{\Omega} (\PP_{\X} + \PP_{\X^{\perp}}) \Delta \|_F \\
        &\le \|\PP_{\Omega} \PP_{\X} \Delta \|_F + \| \PP_{\Omega} \PP_{\X^{\perp}} \Delta \|_F \\
        &\le \frac{1}{2} \| \Delta\|_F + \|\PP_{\X^{\perp}} \Delta \|_F \\
        &\le \frac{1}{2} \|\PP_{\Omega} \Delta\|_F + \frac{1}{2} \|\PP_{\Omega^{\perp}} \Delta\|_F + \|\PP_{\X^{\perp}} \Delta \|_F,
        \end{align*}
        where we have used the facts that $\vertiii{ \PP_{\Omega} \PP_{\X} } \le \frac{1}{2}, \vertiii{ \PP_{\Omega} } \le 1$, and $\vertiii{ \PP_{\X} } \le 1$. Hence,
        \begin{equation*}
        \|\PP_{\Omega} \Delta \|_F \le \|\PP_{\Omega^{\perp}} \Delta\|_F + 2 \|\PP_{\X^{\perp}} \Delta \|_F.
        \end{equation*}
        We now have
        \begin{align*}
        \|\X + \Delta\|_* - \| \X \|_* +  &\lambda (\|\S - \Delta\|_{\nsum} - \| \S \|_{\nsum}) \\
        &> \frac{1}{4} \left( \|\PP_{\X^{\perp}} \Delta \|_* + \lambda \| \PP_{\Omega^{\perp}} \Delta \|_{\nsum} \right) - \frac{\lambda}{8} \big( \|\PP_{\Omega^{\perp}} \Delta\|_F \\
        & \quad \quad + 2 \|\PP_{\X^{\perp}} \Delta \|_F \big) \\
        &\ge \frac{1}{4} \left( 1 - \lambda \right) \| \PP_{\X^{\perp}} \Delta \|_* + \frac{\lambda}{8} \| \PP_{\Omega^{\perp}} \Delta \|_{\nsum}.
        \end{align*}
        The final inequality follows from the fact that for any tensor $\mathcal{T}$, $\|\mathcal{T}\|_{\nsum} \ge \|\mathcal{T}\|_F$ and $\|\mathcal{T}\|_* \ge \|\mathcal{T}\|_F$ (see \eqref{eq:fronuc}). This shows that the perturbation $\Delta$ leads to a strict increase in the objective, unless $\Delta = 0$.
    \end{proof}

In summary, to ensure exact recovery, it suffices to find a tensor $\W^{\perp}$ satisfying
    \begin{equation}
    \label{eq:conditions}
    \begin{cases}
    \PP_{\X^{\perp}} \W^{\perp} = \W^{\perp}, \\
    \|\W^{\perp}\| < \frac{1}{4}, \\
    \| \PP_{\Omega} (\W - \lambda \sgn(\S) + \W^{\perp}) \|_F \le \frac{\lambda}{8}, \\
    \| \PP_{\Omega^{\perp}} (\W + \W^{\perp}) \|_{\max} < \frac{\lambda}{4},
    \end{cases}
    \end{equation}
where we have used \eqref{eq:mainLemma} to state our conditions on $\F$ in terms of $\W^{\perp}$.
    
     Instead of proving the existence of $\W^{\perp}$ directly, we find tensors $\W^\LC$ and $\W^\S$ satisfying $\PP_{\X^{\perp}} (\W^\LC + \W^\S) = \W^\LC + \W^\S$, $\PP_{\Omega} \W^{\S} = \lambda \textnormal{sgn} (\S)$, and the following:
    \begin{equation} 
    \label{condWL}
    \begin{cases}
    (a) \ \ \|\W^\LC\| < \frac{1}{8}, \\
    (b) \ \ \| \PP_{\Omega} (\W + \W^\LC) \|_F \le \frac{\lambda}{8}, \\
    (c) \ \ \| \PP_{\Omega^{\perp}} (\W + \W^\LC) \|_{\max} < \frac{\lambda}{8}, \\
    \end{cases}
    \end{equation}
    
    \begin{equation} 
    \begin{cases}
    \label{condWS}
    (d) \ \ \|\W^\S\| < \frac{1}{8}, \\
    (e) \ \ \| \PP_{\Omega^{\perp}} \W^\S \|_{\max} < \frac{\lambda}{8}.
    \end{cases}
    \end{equation}
    If there exist tensors $\W^\LC$ and $\W^\S$ satisfying these conditions, then the tensor $\W^{\perp} = \W^\LC + \W^\S$ satisfies the conditions of equation \eqref{eq:conditions}, so exact recovery is certain. Similar to the argument in \cite{RPCA}, we construct $\W^\LC$ using a golfing scheme described in Section \ref{WL}, and we construct $\W^\S$ as a the solution to a certain least-squares problem which is outlined in Section \ref{sec:WS}.
    
    \subsection{Constructing $\W^\LC$}
    \label{WL}

    Our construction of $\W^\LC$ uses a variation of the golfing scheme developed in \cite{Gross2,Gross1} and later used in \cite{RPCA,CandesRecht2009,YuanZhang}. Let $n \defeq |\Omega^{\perp}| = d_1 d_2 d_3 - m = (1-\rho_s) d_1 d_2 d_3$. We create an i.i.d.\! uniformly distributed sequence of triples in $[d_1]\times [d_2] \times [d_3]$, call it $\{ (a_i,b_i,c_i) : 1 \le i \le n \}$. This sequence is created by sampling with replacement from $\Omega^{\perp}$ using the following process:
    \begin{enumerate}
        \item Initialize $\mathfrak{S}_0 = \emptyset$.
        \item For $i = 1,2,\cdots,n$, sample the triple $(a_i, b_i, c_i)$ from $\mathfrak{S}_{i-1}$ uniformly with probability $|\mathfrak{S}_{i-1}|/d_1d_2d_3$, and sample $(a_i,b_i,c_i)$ uniformly from $\Omega^{\perp} \backslash {\mathfrak{S}_{i-1}}$ with probability $1-|\mathfrak{S}_{i-1}|/d_1d_2d_3$.
        \item Set $\mathfrak{S}_i = \mathfrak{S}_{i-1} \cup \{ (a_i,b_i,c_i) \}$.
    \end{enumerate}
    Here, each $\mathfrak{S}_i$ is a set containing triples corresponding to indices of the zero elements of $\S$. A similar scheme is used to construct a dual certificate for the tensor completion problem in \cite{YuanZhang} with an important distinction: we are sampling from $\Omega^{\perp}$ to analyze $\S$, while the scheme in \cite{YuanZhang} samples from $\Omega$. This is due to the fact that we observe samples in the complement of $\Omega$, while the opposite is true in the analysis of \cite{YuanZhang}.
    
     Notice that $\mathbb{P}( (a_i,b_i,c_i) \in \mathfrak{S}_{i-1} | \mathfrak{S}_{i-1})$ is equal to the probability of the same event when the triples $(a_i,b_i,c_i)$ are drawn as i.i.d.\! random variables. Also, the conditional distribution of $(a_i,b_i,c_i)$ given $\mathfrak{S}_{i-1}$ and the event $(a_i,b_i,$ $c_i) \in \mathfrak{S}_{i-1}^c$ is uniform. Together, these properties imply that the points $(a_i,b_i,$ $c_i)$ are drawn uniformly from $[d_1] \times [d_2] \times [d_3]$ as i.i.d.\! random variables. Constructing this uniform sample via the golfing scheme is useful because it allows us to split our samples from $\Omega^{\perp}$ into a sequence of independent subsequences.
    
    We split the sequence $\{ (a_i,b_i,c_i) : 1 \le i \le n \}$ into $n_2$ subsets:
    \begin{equation}
        \Omega_k \defeq \{ (a_i,b_i,c_i) : n_1(k-1) \le i \le k n_1 \},
    \end{equation}
    where $| \Omega_k | \le n_1$. Notice that $n_1 n_2 \le n$ due to non-empty intersections among the $\Omega_k$. We choose the values $n_2 = \mathcal{O}\left( \log(n) \right)$ and $n_1 = \mathcal{O} \left( \left( \frac{d_1 d_2 d_3 n}{\mu_0 \log(n)} \right)^{1/2} \right)$. The constants $n_1$ and $n_2$ must be chosen appropriately, and we show why we choose these particular values in later subsections.
    
    With the sets $\Omega_k$ defined, we can define the corresponding projections $\PP_{\Omega_k}$, and use these projections to construct $\W^\LC$. With $\Y_0 = 0$, define the recursive sequence
    \begin{equation*}
    \Y_j = \Y_{j-1} + \frac{d_1 d_2 d_3}{n_1} \PP_{\Omega_j} \PP_{\X} (\W - \Y_{j-1}),
    \end{equation*}
    We set $\W^\LC = \PP_{\X^{\perp}} \Y_{n_2}$, and prove that this choice of $\W^\LC$ satisfies conditions $(a), (b)$, and $(c)$ in \eqref{condWL} with high probability.

\paragraph{Proof of \eqref{condWL}\textup{(\emph{a})}.} 
    Our argument uses ideas from \cite{YuanZhang}. In particular, we use the following two lemmas:

    \begin{lemma}[Lem. 6 in \cite{YuanZhang}]
        \label{specBounds}
        Let $\{ (a_i,b_i,c_i) \}$ be an ordered set of independently and uniformly distributed samples from $[d_1] \times [d_2] \times [d_3]$ and $\Omega_j$ defined as above. Assume that $\coherence(\X) \le \coherence_0$. Define $r \defeq \overline{r}(\X)$. Then for any fixed $j \in \{ 1,2,\cdots, n_2 \}$ and for all $\tau > 0$,
        \begin{align}
        \label{eq:lem4spec}
        &\mathbb{P} \left( \vertiii{ \PP_{\X} - \frac{d_1 d_2 d_3}{n_1} \PP_{\X} \PP_{\Omega_j} \PP_{\X} } \ge \tau \right) \\
        & \quad \quad \le 2 r^2 (d_1+d_2+d_3) \ \exp\left( - \frac{n_1 (\tau^2/2)}{(1 + 2\tau/3) (\coherence_0^2 r^2 (d_1+d_2+d_3)) } \right).
        \end{align}
        Furthermore,
        \begin{align}
        \label{lem4max}
        &\max_{ \| \T \|_{\max} = 1 } \mathbb{P}\left( \left\| \left( \PP_{\X} - \frac{d_1 d_2 d_3}{n_1} \PP_{\X} \PP_{\Omega_j} \PP_{\X} \right) \T \right\|_{\max} \ge \tau \right) \\
        & \quad \quad \le 2 d_1 d_2 d_3 \exp\left( - \frac{n_1 ( \tau^2/2 ) }{(1+2\tau/3) \coherence_0^2 r^2 (d_1+d_2+d_3)} \right).
        \end{align}
    \end{lemma}
    \begin{lemma}[Lem. 7 in \cite{YuanZhang}]
        \label{sparseBounds}
        Let $\alpha(\X) \le \alpha_0$, $r \defeq \overline{r}(\X)$, and $q_1^* = (c + \log(d_1+d_2+d_3))^2 \alpha_0^2 r \log(d_1+d_2+d_3)$. There exists a positive constant $c_1$ so that for any constants $c > 0$ and $\delta_1 \in [1/(\log(d_1+d_2+d_3)), 1)$,
        \begin{equation*}
        n_1 \ge c_1 \left[ q_1^* (d_1+d_2+d_3)^{1+\delta_1} + \sqrt{q_1^*(1+c)\delta_1^{-1} d_1 d_2 d_3} \right]
        \end{equation*}
        implies
        \begin{equation*}
        \max_{\substack{\PP_{\X} \T = \T \\ \|\T\|_{\max} \le \|\W\|_{\max}}} \mathbb{P}\left( \left\|(\PP_{\X} - \frac{d_1d_2d_3}{n_1} \PP_{\X} \PP_{\Omega_1} \PP_{\X}) \T \right\| \ge \frac{1}{16} \right) \le (d_1+d_2+d_3)^{-c - 1},
        \end{equation*}
        where $\W$ is dual to $\X$.
    \end{lemma}
\noindent Lemma \ref{sparseBounds} puts a lower bound on acceptable choices for $n_1$, and because $n_1 n_2 \le n$, this also puts an upper bound on acceptable choices for $n_2$. Our choices of $n_1$ and $n_2$ satisfy these conditions.
    
Instead of working directly with the sequence $\{\Y_j\}$, it is easier to work with the sequence
\begin{equation*}
\Z_j \defeq \W - \PP_{\X} \Y_j.
\end{equation*}
Because $\PP_{\X} \W = \PP_{\X}^0 \W = \W$, it is clear that $\PP_{\X} \Z_j = \Z_j$ for all $j$. From the definition of $\Y_j$, we derive the useful recursion
\begin{align}
\label{eq:recur}
        \Z_j &= \W - \PP_{\X} \left( \Y_{j-1} + \frac{d_1 d_2 d_3}{n_1} \PP_{\Omega_j} \PP_{\X} (\W - \PP_{\X} \Y_{j-1}) \right) \notag \\
        &= \PP_{\X} \left(I - \frac{d_1 d_2 d_3}{n_1} \PP_{\Omega_j} \right) \PP_{\X} \Z_{j-1}.
\end{align}
    We also see that
    \begin{align}
    \Y_{n_2} &= \Y_{n_2-1} + \frac{d_1 d_2 d_3}{n_1}\PP_{\Omega_j} \Z_{n_2 - 1} \notag\\
    & = \Y_{n_2-2} + \frac{d_1 d_2 d_3}{n_1}\PP_{\Omega_j} \Z_{n_2 - 2} + \frac{d_1 d_2 d_3}{n_1}\PP_{\Omega_j} \Z_{n_2 - 1} \notag \\
    & = \sum\limits_j \frac{d_1 d_2 d_3}{n_1} \PP_{\Omega_j} \Z_{j-1}. \label{eq:Ysum}
    \end{align}
    These facts imply
    \begin{align}
    \label{eq:sum}
    \W^\LC &= \PP_{\X^{\perp}} \Y_{n_2} \notag \\ 
    &= \PP_{\X^{\perp}} \sum\limits_j \frac{d_1 d_2 d_3}{n_1} \PP_{\Omega_j} \Z_{j-1} \notag \\
    &= \PP_{\X^{\perp}} \sum\limits_j \left( \frac{d_1 d_2 d_3}{n_1} \PP_{\Omega_j} - I \right) \Z_{j-1},
    \end{align}
    where the last equality follows from the fact that $\PP_{\X} \Z_{j-1} = \Z_{j-1}$. We now use the sequence $\{\Z_j\}$ to show that $\|\W^\LC\| < \frac{1}{8}$. For convenience, let $\Rc_j = I - \frac{d_1 d_2 d_3}{n_1} \PP_{\Omega_j}$. To prove the bound $\|\W^\LC\| < \frac{1}{8}$, we decompose $\W^\LC$ into the sum shown in equation \eqref{eq:sum} and use Lemma \ref{specBounds} to bound the spectral norm of each of the terms.
    \begin{align*}
    \mathbb{P} \Bigg( \| \PP_{\X^{\perp}}  \Y_{n_2} &\| \ge \frac{1}{8} \Bigg) 
    \le\mathbb{P} \left( \left\| \sum_{j=1}^{n_2} \Rc_j \Z_{j-1} \right\| \ge \frac{1}{8} \right) \\
    &\le \mathbb{P} \left( \left\| \Rc_1 \Z_0 \right\| \ge \frac{1}{16} \right) + \mathbb{P}\left( \|\Z_1\|_{\max} \ge \|\W\|_{\max}/4 \right) \\
    &\quad + \mathbb{P}\left( \left\| \sum_{j=2}^{n_2} \Rc_j \Z_{j-1} \right\| \ge \frac{1}{16}, \ \|\Z_1\|_{\max} < \|\W\|_{\max}/4 \right) \\
    &\le \mathbb{P} \left( \|\Rc_1 \Z_0\| \ge \frac{1}{16} \right) + \mathbb{P}\left( \|\Z_1\|_{\max} \ge \|\W\|_{\max}/4 \right) \\
    &\quad + \mathbb{P}\left( \| \Rc_2 \Z_1 \| \ge \frac{1}{32}, \ \|\Z_1\|_{\max} < \|\W\|_{\max}/4 \right) \\
    &\quad + \mathbb{P} \left( \|\Z_2\|_{\max} \ge \|\W\|_{\max}/8, \|\Z_1\|_{\max} < \|\W\|_{\max}/4 \right) \\
    &\quad + \mathbb{P}\left( \left\| \sum_{j=3}^{n_2} \Rc_j \Z_{j-1} \right\| \ge \frac{1}{32}, \|\Z_2\|_{\max} < \|\W\|_{\max}/8 \right) \\
    &\le \mathbb{P} \left( \|\Rc_1 \Z_0\| \ge \frac{1}{16} \right) + \mathbb{P}\left( \|\Z_1\|_{\max} \ge \|\W\|_{\max}/4 \right) \\
    &\quad + \sum_{j=2}^{n_2 - 1} \mathbb{P} \big( \| \PP_{\X^*} \Rc_j \PP_{\X^*} \Z_{j-1} \|_{\max} \ge \|\W\|_{\max}/2^{j+1}, \\
    & \quad \quad \quad \quad \quad \quad \quad \quad \|\Z_{j-1}\|_{\max} \le \|\W\|_{\max}/2^j \big) \\
    &\quad + \sum_{j=2}^{n_2} \mathbb{P}\left( \| \Rc_j \Z_{j-1} \| \ge 2^{-3-j}, \ \|\Z_{j-1}\|_{\max} \le \|\W\|_{\max}/2^{j} \right) \\
    &\le \mathbb{P} \left( \|\Rc_1 \Z_0\| \ge \frac{1}{16} \right) + \mathbb{P}\left( \|\Z_1\|_{\max} \ge \|\W\|_{\max}/4 \right) \\
    &\quad + \sum_{j=2}^{n_2 - 1} \mathbb{P} \left( \| \PP_{\X^*} \Rc_j \PP_{\X^*} \Z_{j-1} \|_{\max} \ge \|\Z_{j-1}\|_{\max}/2 \right) \\
    &\quad + \sum_{j=2}^{n_2} \mathbb{P}\left( \| \Rc_j \Z_{j-1} \| \ge 2^{-3} \|\Z_{j-1}\|_{\max}/\|\W\|_{\max} \right).
    \end{align*}
    For the penultimate inequality, we applied the recursion \eqref{eq:recur}. Because $(\Rc_j,$ $\Z_{j-1})$ are independent, Lemma \ref{specBounds} with the maximizing $\T = \Z_{j-1}/\|\Z_{j-1} \|_{\max}$ and $\tau = \frac{1}{8}$ gives the bound
    \begin{align}
    \mathbb{P} \left( \| \PP_{\X^{\perp}} \Y_{n_2} \| \ge \frac{1}{8} \right) &\le \max_{\substack{\PP_{\X} \T = \T \\ \|\T\|_{\max} \le 1}} \quad \Bigg[ \sum_{j=1}^{n_2 - 1} \mathbb{P} \left( \| \PP_{\X} \Rc_j \PP_{\X} \T \|_{\max} \ge \frac{1}{4} \right) \\
    &\quad \quad \quad + \sum_{j=1}^{n_2} \mathbb{P}\left( \| \Rc_j \T \| \ge \frac{1}{16 \|\W\|_{\max}} \right) \Bigg] \\
    &\le n_2 \Bigg[ \max_{\substack{\PP_{\X} \T = \T \\ \|\T\|_{\max} \le 1}} \quad \mathbb{P}\left( \|\PP_{\X} \Rc_1 \PP_{\X} \T \|_{\max} > \frac{1}{4} \right) \\
    & \quad \quad \quad + \mathbb{P}\left( \|\Rc_1 \T\| > \frac{1}{16 \|\W\|_{\max}} \right) \Bigg] \\
    &\le 2 n_2 d_1 d_2 d_3 \exp\left( - \frac{(3/112) n_1}{ \coherence_0^2 r^2 (d_1+d_2+d_3)} \right) \\
    & \quad \quad \quad + n_2 \left[ \max_{\substack{\PP_{\X} \T = \T \\ \|\T\|_{\max} \le \|\W\|_{\max}}} \mathbb{P}\left( \|\Rc_1 \T\| > \frac{1}{16} \right) \right].
    \end{align}
    Applying Lemma \ref{sparseBounds} to bound the final term, we have shown
    \begin{align}
        \mathbb{P} \Big( \| \PP_{\X^{\perp}} \Y_{n_2} &\| \ge \frac{1}{8} \Big) \\
        &\le 2 n_2 d_1 d_2 d_3 \exp\left( - \frac{(3/112) n_1}{ \coherence_0^2 r^2 (d_1+d_2+d_3)} \right) + (d_1 + d_2 + d_3)^{-1-c}.
    \end{align}
    This shows that $\W^{\LC} = \PP_{\X^{\perp}} \Y_{n_2}$ satisfies condition \eqref{condWL}(a) with high probability.

\paragraph{Proof of \eqref{condWL}\textup{(\emph{b})}.}
    
We would like to prove that $\| \PP_{\Omega} (\W + \W^\LC)\|_F < \frac{\lambda}{8}$. By the definition of the operator norm in \eqref{opNorm}, equation \eqref{eq:lem4spec} of Lemma \ref{specBounds} implies
\begin{equation*}
\left\| \left( \PP_{\X} - \frac{d_1 d_2 d_3}{n_1} \PP_{\X} \PP_{\Omega_j} \PP_{\X} \right) \T \right\|_F \le \tau \| \T \|_F
\end{equation*}
with high probability for all $\T$ satisfying $\PP_{\X} \T = \T$. Consequentially, using the independence of $\Omega_j$ and $\Z_{j-1}$, we have
\begin{equation}
\forall j, \quad \|\Z_j\|_F \le \tau \|\Z_{j-1}\|_F \implies \|\Z_{n_2}\|_F \le \tau^{n_2} \|\W\|_F.
\end{equation}
Using equation \eqref{eq:alphaDef} and 
assuming the coherence of $\X$ satisfies $\alpha(\X)\le \alpha_0$, 
\begin{equation}
\| \W \|_{\max} \le \alpha_0 \left( \frac{\overline{r}}{d_1 d_2 d_3} \right)^{1/2}.
\end{equation}
Therefore, since $\|\X\|_F \le \sqrt{d_1d_2d_3}\|\X\|_{\max}$ for all tensors $\X$ of size $d_1 \times d_2 \times d_3$, 
\begin{equation}
\|\Z_{n_2} \|_F \le \tau^{n_2} \alpha_0 \sqrt{\overline{r}}.
\end{equation}
Hence,
\begin{align}
\| \PP_{\Omega} (\W + \W^\LC) \|_F &= \| \PP_{\Omega} (\W + (I - \PP_{\X}) \Y_{n_2}) \|_F \\
&= \| \PP_{\Omega} \Z_{n_2} \|_F \\
&\le \tau^{n_2} \alpha_0 \sqrt{\overline{r}},
\end{align}
where we used the fact that $\PP_{\Omega} \Y_{n_2} = 0$. Let $\tau = \mathcal{O}(e^{-1})$ and $n_2 = \mathcal{O} \left( \log(n) \right)$. Because $\overline{r} \le \rho_r \left( \frac{n}{(d_1+d_2+d_3) \log(n) \alpha_0^4 \coherence_0^2 } \right)^{1/2}$ (cf.~\eqref{eq:mainResult}), the above bound is smaller than $\frac{\lambda}{8}$ as long as $\rho_r$ 
is a small enough constant.

These parameter choices also ensure the probability that this bound holds, given in Lemma \ref{sparseBounds}, is large. For the sequel, we require that $\|\PP_{\Omega} (\W + \W^\LC)\|_F < \frac{\lambda}{16}$, and it is clear that this bound also holds with high probability.

\paragraph{Proof of \eqref{condWL}\textup{(\emph{c})}.}   
    
    We would like to prove that $\| \PP_{\Omega^{\perp}} (\W + \W^\LC) \|_{\max} < \frac{\lambda}{8}$. We have that $\W + \W^\LC = \Y_{n_2} + \Z_{n_2}$, so $\| \PP_{\Omega^{\perp}} (\W + \W^\LC) \|_{\max} \le \|\Y_{n_2}\|_{\max} + \|\Z_{n_2}\|_{\max}$. From the previous section, we already have the bound $\| \Z_{n_2} \|_{\max} \le \| \Z_{n_2} \|_F \le \frac{\lambda}{16}$, so we must only bound $\| \Y_{n_2} \|_{\max}$.
    \begin{align}
    \|\Y_{n_2}\|_{\max} &= \left( \frac{d_1 d_2 d_3}{n_1} \right) \left\| \sum_{j=1}^{n_2} \PP_{\Omega_j} \Z_{j-1} \right\|_{\max} 
    & \textnormal{(By equation \eqref{eq:Ysum})}\\
    &\le \left( \frac{d_1 d_2 d_3}{n_1} \right) \sum_{j=1}^{n_2} \left\| \PP_{\Omega_j} \Z_{j-1} \right\|_{\max} \\
    &\le \left( \frac{d_1 d_2 d_3}{n_1} \right) \sum_{j=1}^{n_2} \left\| \Z_{j-1} \right\|_{\max} \\
    &\le \left( \frac{d_1 d_2 d_3}{n_1} \right) \left( \sum_{j=0}^{n_2-1} \tau^{j} \right) \left\| \W \right\|_{\max} \\
    &\le \left( \frac{d_1 d_2 d_3}{n_1} \right) \left( \sum_{j=0}^{n_2-1} \tau^{j} \right) \alpha_0 \left( \frac{\overline{r}}{d_1 d_2 d_3} \right)^{1/2} & \textnormal{(By equation \eqref{eq:alphaDef})} \\
    &\le \left( \frac{\alpha_0 \sqrt{ \overline{r} d_1 d_2 d_3}}{n_1} \right) \left( 1 - \tau \right)^{-1}.
    \end{align}
With $n_1 = \mathcal{O}\left(  \left( \frac{d_1 d_2 d_3 n}{\mu_0 \log(n)^2} \right)^{1/2}  \right)$ and $\overline{r} \le \rho_r \left( \frac{n}{(d_1+d_2+d_3) \log(n) \alpha_0^4 \coherence_0^2 } \right)^{1/2}$, it is clear that $\|\Y_{n_2}\| \le \frac{\lambda}{16}$ when $\rho_r$ and $\rho_s$ are small enough, so the desired result holds.
    
These choices of parameters are consistent; it is straightforward to see that $n_1 n_2 \le n$, and that these choices of $n_1$ and $\tau$ ensure that the bounds outlined in Lemma \ref{specBounds} hold with exponentially small probability.

\subsection{Constructing $\W^\S$}
    \label{sec:WS}
    
    In the previous sections, we supposed that the support of $\Omega$ was uniformly distributed over all sets of cardinality $m$. For our construction of $\W^\S$, it is easier to work under the assumption that the support of $\Omega$ follows a Bernoulli distribution with parameter $\rho_s$:
    \begin{equation}
    \label{eq:berdist}
        \Omega_{i,j,k} = \begin{cases}
        1 & \textnormal{w.p. } \rho_s,     \\
        0 & \textnormal{w.p. } 1 - \rho_s. \\
        \end{cases}
    \end{equation}
    We can construct $\W^\S$ under this model and show that $\W^\S$ satisfies conditions \eqref{condWS}(\emph{d}) and \eqref{condWS}(\emph{e}) with high probability. It follows that $\W^\S$ under the original uniform model satisfies conditions \eqref{condWS}(\emph{d}) and \eqref{condWS}(\emph{e}) with high probability as well. This correspondence between the Bernoulli and uniform distributions is well-known and used in \cite{RPCA} as well, although in a slightly different way. We include a proof in Appendix \ref{app:berunif} for completeness.
    
    We make one more simplification to our model before constructing $\W^\S$. Under the model of Theorem \ref{main}, the signs of the entries of $\S$ are arbitrary, but it is more convenient to assume that the signs follow a Bernoulli model with parameter $\frac{\rho_s}{2}$. The equivalence of these two models is also well-known and used in \cite{RPCA}. We save a formal discussion of this equivalence for Appendix \ref{app:berunif}. With these changes to our model in place, we are now prepared to construct $\W^\S$.
    
    Following the ideas behind the construction of the dual certificate for matrix RPCA \cite{RPCA}, we let
    \begin{equation}
    \W^\S = \lambda \PP_{\X^{\perp}} (\PP_{\Omega} - \PP_{\Omega} \PP_{\X} \PP_{\Omega} )^{-1} \sgn(\S).
    \end{equation}
    Our assumption $\vertiii{ \PP_{\Omega} \PP_{\X} } < \frac{1}{2}$ implies $\vertiii{ \PP_{\X} \PP_{\Omega} \PP_{\X} } < \frac{1}{4}$ since these are orthogonal projections, so the inverse of $\PP_{\Omega} - \PP_{\Omega} \PP_{\X} \PP_{\Omega}$ as an operator mapping the range of $\PP_{\Omega}$ onto itself exists. It is also important that $\PP_{\Omega} \W^\S = \lambda \PP_{\Omega} (I - \PP_{\X})(\PP_{\Omega} - \PP_{\Omega} \PP_{\X} \PP_{\Omega} )^{-1} \sgn(\S) = \lambda \sgn(\S)$. As in the matrix case, $\W^\S$ can be interpreted as the tensor with minimum Frobenius norm over the set $\{ \mathcal{T} : \PP_{\X^{\perp}} \mathcal{T} = \T, \ \PP_{\Omega} \T = \lambda \sgn(\S) \}$.

\paragraph{Proof of \eqref{condWS}\textup{(\emph{d})}.}
    
    With $\W^\S = \lambda \PP_{\X^{\perp}} (\PP_{\Omega} - \PP_{\Omega} \PP_{\X} \PP_{\Omega} )^{-1} \sgn(\S)$, we would like to show that $\|\W^\S\| < \frac{1}{8}$. Let $G = \sgn(\S)$ for convenience. The elements of $G$ follow the distribution
    \begin{equation}
    G_{i,j,k} = \begin{cases}
    1 & \textnormal{with probability } \frac{\rho_s}{2}, \\
    0 & \textnormal{with probability } 1 - \rho_s, \\
    -1 & \textnormal{with probability } \frac{\rho_s}{2}.
    \end{cases}
    \end{equation}
    We can then write
    \begin{align}
    \| \W^\S \| &= \lambda \left\| \PP_{\X^{\perp}} (\PP_{\Omega} - \PP_{\Omega} \PP_{\X} \PP_{\Omega} )^{-1} G \right\| \\
    &= \lambda \left\|\PP_{\X^{\perp}} \sum_{k=0}^{\infty} (\PP_{\Omega} \PP_{\X} \PP_{\Omega})^k G \right\| \\
    &\le \lambda \left\|\PP_{\X^{\perp}} G \right\| + \lambda \left\| \PP_{\X^{\perp}} \sum_{k=1}^{\infty} (\PP_{\Omega} \PP_{\X} \PP_{\Omega})^k G \right\| \\
    &\le \lambda \left\| G \right\| + \lambda \left\| \sum_{k=1}^{\infty} (\PP_{\Omega} \PP_{\X} \PP_{\Omega})^k G \right\|
    \label{WS},
    \end{align}
where we have used the Neumann series expansion of the operator $(\PP_{\Omega} - \PP_{\Omega} \PP_{\X} \PP_{\Omega} )^{-1}$. The first term can be bounded using existing tail bounds on the spectral norm of random tensors. The distribution of $G$ is subgaussian, so we can apply the following result from \cite{Tomioka}.
    \begin{lemma}
        \label{lem:SubGaussBound}
        The following holds with probability at least $1-\delta$:
        \begin{equation}
        \|G\| \le \sqrt{8 \left( d_1 + d_2 + d_3 \right) \log(6/ \log(3/2)) + \log (2/\delta)}
        \end{equation}
    \end{lemma}
    \begin{proof}
        This result is an application of \cite[Thm.\ 1]{Tomioka}. See Appendix \ref{app:tomi} for the details.
    \end{proof}
    \noindent With $\lambda = (d_1+d_2+d_3)^{-\frac{1}{2}}$, Lemma \ref{lem:SubGaussBound} ensures that $\lambda \|G\| \le \frac{1}{16}$ with large probability. To bound the second term, we use an $\epsilon$-net covering argument. Define the following set of ``digitalized'' vectors:
    \begin{equation}
    \mathcal{B}_{m_j, d_j} = \{ 0,\pm 1, \pm 2^{-1/2}, \cdots, \pm 2^{-m_j/2} \}^{d_j} \cap \{ u \in \mathbb{R}^{d_j} : \|u\| \le 1 \}.
    \end{equation}
    Let $\Q$ be the operator $\sum_{k=1}^{\infty} (\PP_{\Omega} \PP_{\X} \PP_{\Omega})^k$. We can bound $\|\Q(G)\|$ by considering the action of $\Q$ on tensor product spaces of the $\mathcal{B}_{m_j, d_j}$. With $m_j = \lceil \log_2( d_j ) \rceil$, the following bound holds by Lemma 9 in \cite{YuanZhang}:
    \begin{equation}
    \label{eq:net}
    \|\Q(G)\| = \max_{u_j \in \mathbb{S}^{d_j}} \left< \Q( G ), u_1 \otimes u_2 \otimes u_3 \right> \le 8 \max_{u_j \in \mathcal{B}_{m_j,d_j}} \left< \Q( G ), u_1 \otimes u_2 \otimes u_3 \right>,
    \end{equation}
    where $\mathbb{S}^{d_j}$ is the unit sphere of dimension $d_j$. Because $Q$ is self-adjoint, this also implies
    \begin{equation}
    \| \Q( G ) \| = \max_{u_j \in \mathbb{\S}^{d_j}} \big< \Q(G), u_1 \otimes u_2 \otimes u_3 \big> \le 8 \max_{u_j \in \mathcal{B}_{m_j,d_j}} \left< G, \Q( u_1 \otimes u_2 \otimes u_3 ) \right>.
    \end{equation}
    Let $X(u,v,w) \defeq \big< G, Q( u_1 \otimes u_2 \otimes u_3 ) \big>$ for convenience. Because the signs of $G$ are i.i.d.\! symmetric, we can apply Hoeffding's inequality, conditional on the event that the support of $G$ is exactly $\Omega$:
    \begin{equation}
    \mathbb{P}( | X(u_1,u_2,u_3) | > t ~ \big| ~ \Omega ) \le 2 \exp\left( -\frac{2 t^2}{\| \Q( u_1 \otimes u_2 \otimes u_3 ) \|_F^2} \right).
    \end{equation}
    Because $\|\Q( u_1 \otimes u_2 \otimes u_3 )\|_F^2 \le \vertiii{ \Q }^2 \cdot \|( u_1 \otimes u_2 \otimes u_3 )\|_F^2 = \vertiii{\Q}^2$,
    \begin{equation}
    \mathbb{P}( | X(u_1,u_2,u_3) | > t ~ \big| ~ \Omega ) \le 2 \exp\left( -\frac{2 t^2}{\vertiii{Q}^2} \right).
    \end{equation}
    Taking the maximum of $| X(u_1,u_2,u_3) |$ with $u_j \in \mathcal{B}_{m_j,d_j}, \ j = 1,2,3$, we have
    \begin{equation}
    \mathbb{P} \left( \max_{u_j \in \mathcal{B}_{m_j,d_j}} | X(u_1,u_2,u_3) | > t ~ \big| ~ \Omega \right) \le 2 \left( \prod_{j = 1,2,3} | \mathcal{B}_{m_j,d_j} | \right) \exp\left( -\frac{2t^2}{\vertiii{\Q}^2} \right).
    \end{equation}
    Applying \eqref{eq:net} to this inequality yields
    \begin{equation}
    \mathbb{P}( \| \Q (G) \| > t ~ \big| ~ \Omega ) \le \left( \prod_{j = 1,2,3} | \mathcal{B}_{m_j,d_j} | \right) \exp\left( -\frac{t^2}{32 \vertiii{ \Q }^2} \right).
    \end{equation}
    We can bound $\vertiii{\Q}$ conditional on the event that $\vertiii{ \PP_{\Omega} \PP_{\X} } \le \sigma$. (See Section \ref{sec: boundingTheOp} for a proof that this event holds with high probability.) Recall $\vertiii{\Q} = \vertiii{ \sum_{k=1}^{\infty} (\PP_{\Omega} \PP_{\X} \PP_{\Omega})^k } \le \sum_{k=1}^{\infty} \vertiii{ ((\PP_{\Omega} \PP_{\X}) (\PP_{\Omega} \PP_{\X})^*)^{k} } \le \frac{\sigma^2}{1-\sigma^2}$. We can also bound the cardinality of $\mathcal{B}_{m_j,d_j}$  using equation (21) of \cite{YuanZhang}:
    \begin{equation}
        \prod_{j = 1,2,3} | \mathcal{B}_{m_j,d_j} | = e^{(21/4) (d_1 + d_2 + d_3)}.
    \end{equation}
    This gives us the unconditional bound
    \begin{align}
    \label{eq:12d}
    \mathbb{P}( & \lambda \| Q (G) \| > t ) \\
    &\le 2 \left( \prod_{j = 1,2,3} | \mathcal{B}_{m_j,d_j} | \right) \exp\left( -\frac{t^2 (1-\sigma^2)^2}{32 \sigma^4 \lambda^2} \right) + \mathbb{P}( \vertiii{ \PP_{\Omega} \PP_{\X} } > \sigma ) \notag \\
    &\le 2 \left( \exp \left(\frac{21 (d_1+d_2+d_3)}{4} \right) \right) \exp\left( -\frac{t^2 (1-\sigma^2)^2}{32 \sigma^4 \lambda^2} \right) + \mathbb{P}( \vertiii{ \PP_{\Omega} \PP_{\X} } > \sigma ) \notag \\
    &= 2 \exp\left( -\frac{t^2 (1-\sigma^2)^2}{32 \sigma^4 \lambda^2} + \frac{21 (d_1+d_2+d_3)}{4} \right) + \mathbb{P}( \vertiii{ \PP_{\Omega} \PP_{\X} } > \sigma ).
    \end{align}
    This shows that if $\lambda$ is chosen to be on the order of $(d_1+d_2+d_3)^{-1/2}$, and if $\sigma$ is a small-enough constant, then $\|\W^\S\| < \frac{1}{8}$ with high probability.

\paragraph{Proof of \eqref{condWS}\textup{(\emph{e})}.}
    
    We would like to show that $\| \PP_{\Omega^{\perp}} \W^\S \|_{\max} < \frac{\lambda}{8}$. By our definition of $\W^\S$,
    \begin{align}
    \W^\S &= \lambda \PP_{\X^{\perp}} (\PP_{\Omega} - \PP_{\Omega} \PP_{\X} \PP_{\Omega})^{-1} G \\
    &= \lambda (I - \PP_{\X}) (\PP_{\Omega} - \PP_{\Omega} \PP_{\X} \PP_{\Omega})^{-1} G \\
    &= - \lambda \PP_{\X} (\PP_{\Omega} - \PP_{\Omega} \PP_{\X} \PP_{\Omega})^{-1} G
    \end{align}
    Choosing $(i,j,k) \in \Omega^\perp$, we have
    \begin{align}
    \W^\S_{i,j,k} &= \left< \W^\S, e_i \otimes e_j \otimes e_k \right> \\
    &= - \lambda \left< \PP_{\X} (\PP_{\Omega} - \PP_{\Omega} \PP_{\X} \PP_{\Omega} )^{-1} G, e_i \otimes e_j \otimes e_k \right> \\
    &= - \lambda \left< \PP_{\X} (\PP_{\Omega} - \PP_{\Omega} \PP_{\X} \PP_{\Omega} )^{-1} G, e_i \otimes e_j \otimes e_k \right> \\
    &= -\lambda \left< G, (\PP_{\Omega} - \PP_{\Omega} \PP_{\X} \PP_{\Omega} )^{-1} \PP_{\Omega} \PP_{\X} (e_i \otimes e_j \otimes e_k) \right>,
    \end{align}
    where we used the fact that $\PP_{\Omega}, \ \PP_{\X}$, and $(\PP_{\Omega} - \PP_{\Omega} \PP_{\X} \PP_{\Omega} )^{-1}$ are self-adjoint. For convenience, let
    \begin{equation}
    R(i,j,k) = (\PP_{\Omega} - \PP_{\Omega} \PP_{\X} \PP_{\Omega} )^{-1} \PP_{\Omega} \PP_{\X} (e_i \otimes e_j \otimes e_k).
    \end{equation}
    Now we bound the maximum entry of $\W^\S$ with high probability conditional on the events that the support of $G$ is exactly $\Omega$ and that $\vertiii{ \PP_{\Omega} \PP_{\X} } \le \sigma$. Because the entries of $G$ are i.i.d.\! symmetric, Hoeffding's inequality gives
    \begin{equation}
    \mathbb{P} \left(\left| \W^\S_{i,j,k} \right| \ge \eta ~\big|~ \Omega \right) \le 2 \exp \left(- \frac{2\eta^2}{\lambda^2 \| R(i,j,k) \|_F^2} \right),
    \end{equation}
    so
    \begin{equation}
    \mathbb{P} \left( \max_{i,j,k} \left| \W^\S_{i,j,k} \right| \ge \eta ~\big|~ \Omega \right) \le 2 d_1 d_2 d_3 \exp \left(- \frac{2\eta^2}{\lambda^2 \max\limits_{i,j,k} \| R(i,j,k) \|_F^2} \right).
    \end{equation}
    All that is left is to bound $\|R(i,j,k)\|_F^2$. We need the following lemma from \cite{YuanZhang}:
    \begin{lemma}[Lem. 2 in \cite{YuanZhang}]
\label{coherLemma}

Let $\X \in \mathbb{R}^{d_1 \times d_2 \times d_3}$ be a third-order tensor. Then
\begin{equation}
\max_{i,j,k} \|\PP_{\X} (e_i \otimes e_j \otimes e_k) \|^2_F \le \frac{\overline{r}(\X)^2 (d_1+d_2+d_3)}{d_1 d_2 d_3} \mu(\X)^2.
\end{equation}
\end{lemma}
By Lemma \ref{coherLemma},
    \begin{equation}
    \| \PP_{\X} (e_i \otimes e_j \otimes e_k) \|_F \le \coherence_0 \overline{r} \left( \frac{d_1+d_2+d_3}{d_1d_2d_3}\right)^{1/2}.
    \end{equation}
    Hence,
    \begin{align}
    \| \PP_{\Omega} \PP_{\X} (e_i \otimes e_j \otimes e_k) \|_F &\le (\vertiii{\PP_{\Omega} \PP_{\X} } ) ( \|\PP_{\X} (e_i \otimes e_j \otimes e_k)\|_F) \\
    &\le \sigma \coherence_0 \overline{r} \left( \frac{d_1+d_2+d_3}{d_1d_2d_3}\right)^{1/2}
    \end{align}
    Furthermore,
    \begin{align}
    \| (\PP_{\Omega} - \PP_{\Omega} \PP_{\X} \PP_{\Omega} )^{-1} \| &= \| (\PP_{\Omega} - (\PP_{\Omega} \PP_{\X}) (\PP_{\Omega} \PP_{\X})^*  )^{-1} \| \\
    &\le (1-\sigma^2)^{-1}.
    \end{align}
    Combining these two bounds, we have
    \begin{align}
    \| R(i,j,k) \|_F^2 &= \| (\PP_{\Omega} - \PP_{\Omega} \PP_{\X} \PP_{\Omega} )^{-1} \PP_{\Omega} \PP_{\X} (e_i \otimes e_j \otimes e_k) \|_F^2 \\
    &\le \left( \frac{\sigma \coherence_0 \overline{r}}{1-\sigma^2} \right)^2 \left( \frac{d_1+d_2+d_3}{d_1d_2d_3}\right).
    \end{align}
    Finally, we have derived that
    \begin{align}
    \label{eq:Firste}
    \mathbb{P} \left( \max_{i,j,k} \left| \W^\S_{i,j,k} \right| \ge \eta \right) \le & 2 d_1 d_2 d_3 \exp \left(- \frac{2\eta^2(1-\sigma^2)^2 d_1 d_2 d_3}{(\sigma \lambda \coherence_0 \overline{r})^2 (d_1 + d_2 + d_3)} \right) \notag \\
    & \quad + \mathbb{P}\left( \vertiii{ \PP_{\Omega} \PP_{\X} } > \sigma \right).
    \end{align}

    Letting $\eta = \frac{\lambda}{8}$, this shows that if $\overline{r} \le \rho_r \left( \frac{n}{(d_1 + d_2 + d_3)\log(n) \alpha_0^4 \coherence_0^2 } \right)^{1/2}$, then $\|\PP_{\Omega^{\perp}} \W^\S\|_{\max} < \frac{\lambda}{8}$ with high probability, provided that $\rho_r$ is a small-enough constant.

    \subsection{Bounding the Operator Norm of $\PP_{\Omega} \PP_{\X}$}
    \label{sec: boundingTheOp}
    
    The past two results were conditional on the event that $\vertiii{ \PP_{\Omega} \PP_{\X} } \le \sigma$ with high probability. Lemma 5 from \cite{YuanZhang} shows this is true.
    
    \begin{lemma}[Lem. 5 in \cite{YuanZhang}]
        \label{lem:opBound}

        Let $\coherence(\X) \le \coherence_0$, $\overline{r}(\X) = r$, and $\Omega$ follow the Bernoulli model of equation \eqref{eq:berdist}. Then, for any $\tau > 0$
        \begin{align}
        \mathbb{P}\Big( &\vertiii{ \PP_{\X} (\tfrac{d_1d_2d_3}{n} \PP_{\Omega} - I) \PP_{\X} } \ge \sigma \Big) \\
        & \le 2r^2 (d_1+d_2+d_3) \exp\left(- \frac{n \sigma^2/2}{(1+2\sigma/3) \coherence_0^2 r^2 (d_1+d_2+d_3)} \right).
        \end{align}
    \end{lemma}
    
    By our assumption in \eqref{eq:mainResult}, $r \le \mathcal{O} \left( \left(\frac{n}{(d_1+d_2+d_3) \log(n) \alpha_0^4 \mu_0^2} \right)^{\frac{1}{2}} \right)$, so as long as $n$ is not too large (or, equivalently, $\rho_s$ is sufficiently small) Lemma \ref{lem:opBound} shows that as an operator in the range of $\PP_{\X}$, $\tfrac{d_1 d_2 d_3}{n} \vertiii{\PP_{\X} \PP_{\Omega} \PP_{\X} } \in [1/2,3/2]$. Using the fact that $\vertiii{ \PP_{\Omega} \PP_{\X} }^2 = \vertiii{ (\PP_{\Omega} \PP_{\X})^* (\PP_{\Omega} \PP_{\X}) } = \vertiii{\PP_{\X} \PP_{\Omega} \PP_{\X} }$, we have also bounded $\vertiii{ \PP_{\Omega} \PP_{\X} }$ with high probability.
    
    Lemma \ref{lem:opBound} assumes the support of $\Omega$ is uniformly distributed over all sets of cardinality $m$, but in Section \ref{sec:WS}, we assume that $\Omega$ follows the Bernoulli model. In Appendix \ref{app:berunif}, we show that these two models are essentially equivalent, so the conclusion of Lemma \ref{lem:opBound} holds for the Bernoulli model as well.

\section{A Nonconvex Approach to Atomic Norm Minimization}
\label{sec:NonCon}

Although Theorem \ref{main} shows that the program \eqref{RPCA} can exactly recover a low Tucker-rank tensor and a sparse tensor from their superposition, \eqref{RPCA} is NP-hard to solve in general due to the intractability of the atomic norm \cite{NPtensors}. For low-rank matrix recovery, it is common to accelerate computation by replacing the nuclear norm with a nonconvex formulation based on a factorization of $X$ \cite{CandesRecht2009,fazel2001rank,MMMF2005},
\begin{equation}
\label{nonconnuc}
\|X\|_* = \inf_{UV^T = X} \quad \tfrac{1}{2} (\|U\|_F^2 + \|V\|_F^2).
\end{equation}
A similar factorized formulation of the tensor atomic norm has been used for order-three tensors in \cite{Bazerque} and \cite{MomentTensorAtom} for tensor completion. Due to the non-smooth regularizer in the tensor RPCA problem, this factorization approach must be handled with care. We introduce our factorized program in section \ref{subsec:burmontfact} and discuss a particularly fast method for handling the non-smooth regularizer in section \ref{sec:firstOrd}. We conclude this section by showing that many local minima of our factorized program are globally optimal.

\subsection{Burer-Monteiro Factorization in Higher-Orders}
\label{subsec:burmontfact}

We now explicitly work with tensors of order-$K$, with $K \ge 3$. Instead of establishing equivalence to the program \eqref{RPCA} directly, we work with its Lagrangian formulation:
\begin{align}
\label{RPCAlag}
\min_{\X, \S} \quad &\frac{1}{2} \| \X + \S - \Z \|_F^2 + \lambdaX \| \X \|_* + \lambdaS \| \S \|_{\textnormal{sum}}. \end{align}
A solution $(\X^*,\S^*)$ to \eqref{RPCAlag} solves a variant of \eqref{RPCA} where the equality constraint is replaced with $\| \X + \S - \Z \|_F \le \epsilon$ with $\epsilon \defeq \| \X^* + \S^* - \Z \|_F$ and $\lambda = \lambdaS/\lambdaX$.

Our factorized approach implicitly introduces a bound on the rank of $\X$, yielding the constrained problem
\begin{equation}
\label{RPCAlagCon}
\min_{\X, \S} \quad \frac{1}{2} \| \X + \S - \Z \|_F^2 + \lambdaX \| \X \|_* + \lambdaS \| \S \|_{\textnormal{sum}} \quad \textnormal{s.t.} \quad \textnormal{rank}_{\textnormal{atomic}}(\X) \le R.
\end{equation}
Although \eqref{RPCAlagCon} is nonconvex, it has the same global optima as the convex program \eqref{RPCAlag} as long as the rank-bound $R$ is non-restrictive at the solution \cite{GLRM}. 
We can explicitly parameterize the nonconvex model \eqref{RPCAlagCon} as 
\begin{align}
\label{noncon1}
\min_{\{ a_r^{(1)} \}, \cdots, \{ a_r^{(K)} \},\S} \quad \frac{1}{2} \Bigg\| \sum_{r=1}^R (a_r^{(1)} \otimes \cdots \otimes a_r^{(K)}) & + \S - \Z \Bigg\|_F^2 \notag \\
& + \lambdaS \| \S \|_{\textnormal{sum}} + \frac{\lambdaX}{K} \sum_{r=1}^R \sum_{k = 1}^K \|a_r^{(k)}\|^K.
\end{align}
The following proposition shows that this nonconvex program is equivalent to \eqref{RPCAlag} as long as the induced rank-bound is non-restrictive at the solution.
\begin{proposition}
\label{prog2}
Suppose $\{ a_r^{* (1)} , \cdots,  a_r^{*(K)} \}_{r=1,\ldots,R}, S^*$ 
are optimal for the nonconvex program \eqref{noncon1}, and define $\X^* \defeq \sum_{r=1}^R (a_r^{* (1)} \otimes \cdots \otimes a_r^{* (K)})$. 
Then the point $(\X^*,S^*)$ is optimal for the problem \eqref{RPCAlag}. Conversely, if $(\X^*,\S^*)$ is the minimizer of \eqref{RPCAlag}, then the terms 
  $\{ a_r^{* (1)} \}, \cdots, \{ a_r^{*(K)} \}$
   from an atomic decomposition of $\X^*$ are optimal for \eqref{noncon1}.
  
\end{proposition}
The authors of \cite{Bazerque} and \cite{MomentTensorAtom} prove results similar to Proposition \ref{prog2} in the case $K = 3$, and we include a proof in Appendix~\ref{sec:proofBurerMonteiro} for completeness. 

The factorized regularizer in \eqref{noncon1} can be viewed as a higher-order generalization of Burer-Monteiro factorization, which is popular in low-rank matrix recovery \cite{CandesRecht2009,fazel2001rank,MMMF2005,GLRM}. Notice for $K = 2$, this term reduces to $\frac{\lambdaX}{2} (\|A\|_F^2 + \|B\|_F^2)$, which is equivalent to standard nuclear-norm regularization (as shown in \eqref{nonconnuc}). The next section discusses how we can efficiently find a stationary point of \eqref{noncon1}.

\subsection{First-Order Solvers}
\label{sec:firstOrd}
\newcommand{\allA}{\mathbf{a}}

One of the drawbacks of \eqref{noncon1} is that the non-smooth $\ell_1$-regularizer $ \|S\|_{\textnormal{sum}}$  prevents the use of pure gradient-based solvers. 
One can use proximal gradient methods but these can still be slow. 
However, as suggested in \cite{Us1}, the structure of the program allows us to smooth the problem through marginalization. Define
\begin{equation}
\label{phi}
\varphi : \X \mapsto \min_S \frac{1}{2} \|\X + S - \Z\|_F^2 + \lambdaS \|S\|_{\textnormal{sum}}.
\end{equation}
Because we are using the least-squares loss and $\ell_1$-regularization, $\varphi$ is the (shifted) Moreau envelope of the $\ell_1$-norm, also known as the Huber loss \cite{LevelSet}. The objective in \eqref{phi} is strongly convex, so the minimum exists and is unique. In fact, it can be written in closed-form using the shrinkage operator:
\begin{equation*}
(\textnormal{shrink}(Y,\lambda))_{i_1,\cdots,i_K} \defeq \textnormal{sign}(Y_{i_1,\cdots,i_K}) \lfloor \,|Y_{i_1,\cdots,i_K}| - \lambdaS | \,\rfloor_+ ,
\end{equation*}
(where $\lfloor a\rfloor_+$ denotes the non-negative part of $a$), so that
\begin{equation}
\argmin_S \frac{1}{2} \|\X + S - \Z\|_F^2 + \lambdaS \|S\|_{\textnormal{sum}} = \textnormal{shrink}(\Z-\X,\lambdaS).
\end{equation}
Incorporating $\varphi$ into the convex program \eqref{RPCA} preserves convexity and introduces differentiability. Combining $\varphi$ with the nonconvex program \eqref{noncon1} yields a tractable, Lipschitz-differentiable problem that is amenable to first-order solvers. Formally,
letting $\allA = ( a_r^{(k)} )_{r=1,\ldots,R}^{k=1,\ldots,K}$, 
we solve the program
\newcommand{\vecToTens}{\mathcal{A}}
\begin{equation}
\label{marg}
\min_{\allA}
\quad \frac{\lambdaX}{K} \sum_{r=1}^R \sum_{k = 1}^K \| a_r^{(k)} \|^K + 
\varphi \left( \vecToTens(\allA) \right),
\end{equation}
where
\begin{equation}
\vecToTens : \allA \mapsto \sum_{r=1}^R (a_r^{(1)} \otimes \cdots \otimes a_r^{(K)})
\end{equation}
Let $f(\allA)$ be the objective in \eqref{marg}, and let $S^{\allA}$ be the minimizer in \eqref{phi} (which implicitly depends on $\allA$). Then $f$ is differentiable with gradient given by 
\begin{align*}
\nabla_{a_r^{(k)}} f &= (\lambdaX \|a_r^{(k)}\|^{K-2} ) a_r^{(k)} + \nabla_{a_r^{(k)}} 
\psi( \allA ) \Big|_{S^{\allA}} \\
&= (\lambdaX \|a_r^{(k)}\|^{K-2} ) a_r^{(k)} + \left( a_r^{(k)} C^T + S^{\allA} - \Z \right) C,
\end{align*}
where $C = \left( A^{(K)} \odot \cdots \odot A^{(k+1)} \odot A^{(k-1)} \odot \cdots \odot A^{(1)} \right)$. We direct the reader to \cite[Thm.\ 10.58]{VariationalAnalysis} for proof of the first equality and to \cite{KoldaGrad} for a derivation of the second.

\subsection{Global Optimality of Certain Local Minima}

Although it is possible to find a stationary point of \eqref{noncon1} efficiently, it is not obvious that this stationary point approximates the global optimizer when $K \ge 3$. This contrasts with the case $K = 2$, where several works have shown that the local optima of \eqref{noncon1} or similar models are all globally optimal, and non-optimal stationary points are avoidable.

The case $K \ge 3$ has received considerably less attention. However, we can use a recent result of \cite{HaeffeleVidal15} to show that certain local minima of \eqref{noncon1} are globally optimal. The following proposition is a special case of \cite[Thm.\  15]{HaeffeleVidal15}.

\begin{proposition}[\cite{HaeffeleVidal15}]
\label{prop:loc2glob}
Let $\ell(\X,\S)$ be once differentiable and jointly convex in $\X$ and $\S$, and let $R(\S)$ be convex but not necessarily differentiable. Any local minimizer of the optimization problem
\begin{equation}
    \min_{\{ a_r^{(1)} \}, \cdots, \{ a_r^{(K)} \},\S} \quad \ell \left( \sum_{r=1}^R (a_r^{(1)} \otimes \cdots \otimes a_r^{(K)}), \S \right) + R( \S ) + \sum_{r=1}^R \sum_{k = 1}^K \|a_r^{(k)}\|^K.
\end{equation}
such that $a_{r_0}^{(k_0)} = 0$ for some $r \in \{ 1,\cdots,R \}$ and $k \in \{1,\cdots,K\}$ is a global minimizer.
\end{proposition}

Clearly, Proposition \ref{prop:loc2glob} applies to the model in \eqref{noncon1}. The condition that one of the factors $a_{r_0}^{k_0}$ is equal to zero is a natural condition for global optimality because it implies that the rank-bound the factorization induces is not too strict. Proposition \ref{prop:loc2glob} suggests that if $R$ is large enough, then the stationary points of \eqref{noncon1} are good approximations of the global optimizers. This idea is discussed formally in \cite{HaeffeleVidal15}, and we direct the interested reader there for a discussion of descent approaches that rigorously converge to global optimizers.

\section{Tensor RPCA for Topic Modeling}
\label{sec:topmod}

In \cite{AnandHsu}, the authors demonstrate that training many latent variable models can be reduced to a tensor decomposition problem. Current methods for obtaining this decomposition are generally based on the Tensor Power Method described in \cite{AnandHsu}, and they all require the tensor being decomposed to have orthogonal factor matrices \cite{sketchTensor,RongTensor,AnandHsu}. To meet this requirement, the underlying distributions must be linearly independent, and a numerically unstable whitening procedure must be implemented before the tensor decomposition takes place.

Tensor RPCA avoids these problems because it does not assume the factor matrices are orthogonal. Furthermore, it is also robust to sparsely distributed errors in the sample moment-tensor, which often arise from systematic errors in the sample moments \cite{AnandRPCA}. In the following subsections, we outline how tensor RPCA can be used for topic modeling, specifically, for parameter estimation in Latent Dirichlet Allocation (LDA). Central to this approach is how Theorem \ref{main} interpreted in this context provides provable performance guarantees, stating that as long as the number of errors in the sample moments is small, and the number of topics is small with respect to the vocabulary size, then tensor RPCA perfectly recovers the topic distributions with high probability.

\subsection{Problem Setting and Existing Approaches}
\label{sec:topic}

In the LDA model, every document is a mixture of topics, and these mixtures follow a Dirichlet distribution Dir$(\beta)$. Here, $\beta \in \R^K_{++}$ is a vector of $K$ model parameters, where $K$ is the number of topics. The probability density of this distribution over the simplex $\Delta^{K-1}$ is given by \cite{LDA}
\begin{equation}
p_{\beta}(h) = \frac{\Gamma(\beta_0)}{\prod_{i=1}^K \Gamma(\beta_i)} \prod_{i=1}^K h_i^{\beta_i - 1}, \quad h \in \Delta^{K-1}, \quad \beta_0 := \sum_{i=1}^K \beta_i.
\end{equation}

To form a document under this model, we first draw a topic mixture $h = (h_1,\cdots,h_K) \sim$ Dir$(\beta)$, and conditioned on this mixture, we independently draw $\ell$ words $w_1,\cdots, w_{\ell}$ from the distribution $\sum_{i=1}^K h_i \nu_i$, where $\nu_i \in \Delta^{d-1}$ represents the distribution over the vocabulary corresponding to the $i^{th}$ topic. The words $w_i$ follow a one-hot vector encoding, so $w_i = e_i$ (where $e_i$ is a canonical basis vector) if and only if the $i^{th}$ word of the document is $w_i$. The following theorem relates low-order moments of this model to the topic distributions:
\begin{theorem}[\cite{Anand1,AnandHsu}]
\label{Anand}

Let
\begin{align*}
M_1 &\defeq \mathbb{E}[w_1], \quad \quad M_2 \defeq \mathbb{E}[w_1 \otimes w_2] - \frac{\beta_0}{\beta_0 + 1} M_1 \otimes M_1, \\
\mathcal{M}_3 & \defeq \mathbb{E}[w_1 \otimes w_2 \otimes w_3] - \frac{\beta_0}{\beta_0 + 2} ( \mathbb{E}[w_1 \otimes w_2 \otimes M_1] + \mathbb{E}[w_1 \otimes M_1 \otimes w_2] \\
& \quad \quad + \mathbb{E}[M_1 \otimes w_1 \otimes w_2] ) + \frac{2\beta_0^2}{(\beta_0 + 2)(\beta_0+1)} M_1 \otimes M_1 \otimes M_1.
\end{align*}
Then
\begin{equation}
\label{eq:moments}
M_2 = \sum_{i=1}^K \frac{\beta_i}{(\beta_0+1)\beta_0} \nu_i \otimes \nu_i, \quad \quad \mathcal{M}_3 = \sum_{i=1}^K \frac{2\beta_i}{(\beta_0+2)(\beta_0+1)\beta_0} \nu_i \otimes \nu_i \otimes \nu_i
\end{equation}
\end{theorem}

Hence, the topic distributions $\nu_i$ can be discovered from the rank-one decompositions of $M_2$ and $\mathcal{M}_3$; decomposition of $M_2$ alone is not sufficient since there is not a \emph{unique} rank-$K$ decomposition of a symmetric matrix, whereas the tensor CP decomposition is often unique, cf.~Thm.~\ref{thm:unique}. As described in \cite{AnandHsu}, this type of structure exists in the low-order moments of numerous latent variable models. The results presented in this paper extend to these problems as well.

\subsection{Theorem \ref{main} for Provable Topic Modeling}

In the context of topic modeling, we would like to recover the population moment-tensor $\mathcal{M}_3$ in \eqref{eq:moments} and its rank-one decomposition even though the empirical moment-tensor has entries with large errors. Theorem \ref{main} shows that exactly recovering the population moment-tensor is possible when the number of topics and the number of errors are bounded.
\begin{corollary}
\label{mainTop}
Let $\Z \in \mathbb{R}^{d \times d \times d}$ be the empirical third-order moment-tensor. Suppose the population moment-tensor $\mathcal{M}_3 \in \mathbb{R}^{d \times d \times d}$ satisfies $\mu(\mathcal{M}_3) \le \mu_0$, and that the true number of topics is $\rank_{\textnormal{CP}}(\mathcal{M}_3) = K$. Let tensor $\S \in \mathbb{R}^{d \times d \times d}$ be the tensor of discrepancies\footnote{
In practice, a robust version of RPCA, much like the Lagrangian formulation used in \eqref{RPCAlag}, can be used, which distinguishes ubiquitious small-magnitude discrepancies from rare but large discrepancies.
There are matrix RPCA results for this case which cannot promise exact recovery, but do guarantee recovery up to the level of the small-magnitude noise.
} between the empirical and population moment-tensors. Suppose $\S$ has a support set $\Omega$ that is uniformly distributed among all sets of cardinality $m$, and let $n = d^3 - m$. Then there exists a positive constant $c$ so that tensor RPCA with $\lambda = (3d)^{-1/2}$ exactly recovers $\mathcal{M}_3$ and $\S$ with probability $1-d^{-1-c}$, provided that
\begin{equation}
K \le \rho_r \left( \frac{n}{3d \log( n ) \alpha_0^4 \mu_0^2} \right)^{1/2} \quad \textnormal{\textit{and}} \quad m \le \rho_s d^3.
\end{equation}
\end{corollary}
\begin{proof}
        Recall that $\rank_{\textnormal{CP}}(\mathcal{M}_3) \in [\overline{r}(\mathcal{M}_3),\overline{r}(\mathcal{M}_3)^2]$, so our assumption implies $\overline{r}(\mathcal{M}_3) \le \rho_r \left( \frac{n}{3d \log( n ) \alpha_0^4 \mu_0^2} \right)^{1/2}$. Applying Theome \ref{main} proves the result.
\end{proof}

Hence, if the number of topics is much less than the size of our vocabulary, $d$, then the actual third-order moment can be exactly recovered.

Of course, this result says nothing about the recovery of rank-one components of the moment-tensor $\mathcal{M}_3$, which are what reveal the topic distributions. Also, it is impractical to work with tensors of size $d \times d \times d$, as the vocabulary size is generally extremely large. We address each of these problems in the following sections.

\subsection{Identifiability}

By solving the nonconvex formulation of tensor RPCA \eqref{noncon1}, we implicitly solve for a rank-one decomposition of the moment-tensor. However, if this CPD is not unique, it is unclear whether the recovered rank-one factors correspond to the topic distributions. Fortunately, a tensor's CPD is unique under mild conditions.
\begin{theorem}
\label{thm:unique}
\cite{KruskalUnique} Let $k_{\nu}$ be the maximum value such that the vectors in any subset of $\{\nu_r\}_{r=1}^K$ of size $k_{\nu}$ are linearly independent. If $3 k_{\nu} \ge 2K + 2$, then the CPD of $\mathcal{M}_3 = \sum_{r=1}^R \nu_r \otimes \nu_r \otimes \nu_r$ is unique up to permutation and scaling of the topic distributions $\nu_r$.
\end{theorem}
The condition in Theorem \ref{thm:unique} is weaker than requiring the distributions to be linearly independent, which is necessary for most existing approaches \cite{Anand1,AnandHsu}. Further information on using Kruskal's theorem to establish identifiability results in latent structure models can be found in \cite{AllmanIdent}.

For any CPD, it is possible to form an equivalent CPD by permuting and rescaling the columns of the factor matrices. Precisely, if some tensor $\X = \sum_{r=1}^R a_r \otimes b_r \otimes c_r$, then $\X = \sum_{r=1}^R k_1 a_{\pi(r)} \otimes k_2 b_{\pi(r)} \otimes k_3 c_{\pi(r)}$ as well, where $\pi$ is any permutation on the set $\{ 1,2,\cdots R \}$ and $k_1 k_2 k_3 = 1$. However, neither of these operations affect the topic distributions found from $\mathcal{M}_3$ because we constrain each rank-one factor of $\mathcal{M}_3$ to be symmetric (i.e., of the form $\nu_r \otimes \nu_r \otimes \nu_r)$. This restriction disallows rescaling. Permuting the topic distributions does not change the distributions, so equivalence up to scaling and permutation still ensures that the topic distributions are well-defined.

\subsection{Dimensionality Reduction, Whitening, and Oversampling}
\label{sec:whiten}

It is often not feasible to decompose tensors of size $d \times d \times d$ when the vocabulary is large. To overcome this, existing works apply a dimensionality reduction technique so that the topic distributions can be recovered from the rank-one factors of a smaller tensor \cite{Anand1,AnandHsu}. This dimensionality reduction step is closely related to the whitening procedure presented in \cite{Anand1,AnandHsu} that is used to orthogonalize the factor matrices of a tensor so that orthogonal-decomposition algorithms, such as the Tensor Power Method, are applicable. In this section, we outline the whitening procedure, discuss its numerical instability, and show how a similar but stable procedure can be used for dimensionality reduction in tensor RPCA.

Suppose the topic distributions we seek are linearly independent. Let $M_2$ be the empirical second-order moment in \eqref{eq:moments}, and let $M_2 =: U \Sigma V^T$ be its (skinny) singular value decomposition. Define the whitening matrix as $W \defeq \Sigma^{\dagger} U^T \in \R^{K \times d}$, where $\Sigma^{\dagger}$ is the pseudo-inverse of $\Sigma$. The whitened third-order empirical moment is then \cite{Anand1,AnandHsu}
\begin{equation}
\widetilde{\mathcal{M}}_3 \defeq (W,W,W) \cdot \mathcal{M}_3 = \sum_{r=1}^K \lambda_r (W \nu_i \otimes W \nu_i \otimes W \nu_i ) \in \R^{K\times K \times K}
\end{equation}
The components $W\nu_i$ are orthogonal and can be found by decomposing $\widetilde{\mathcal{M}}_3 \in \R^{K\times K \times K}$, which is much smaller than the original tensor $\mathcal{M}_3 \in \R^{d\times d \times d}$. The factors of $\mathcal{M}_3$ can then be found by applying the inverse whitening transform, $\mathcal{M}_3 = (W^{\dagger},W^{\dagger},W^{\dagger}) \cdot \widetilde{\mathcal{M}}_3$, where $W^{\dagger} = U S$. Multiplying by the pseudo-inverse $W^{\dagger}$ is numerically unstable if $W$ has a large condition number, which is common for real-world data sets, and this introduces unnecessary error into the computation.

For tensor RPCA, we can perform the same dimensionality reduction without whitening. Tensor RPCA does not require the tensor to be orthogonally decomposable, so we can drop the assumption that the topic distributions are linearly independent. The transformation $Q \defeq I_{K \times d} U^T$ performs the same dimensionality reduction as the whitening transform: $\widehat{\mathcal{M}}_3 \defeq (Q,Q,Q) \cdot \mathcal{M}_3 \in \R^{K \times K \times K}$. However, $Q$ does not orthogonalize the factor matrices of $\mathcal{M}_3$ as this is not required, and $Q$ is perfectly conditioned with condition number 1. The factors of $\mathcal{M}_3$ can be recovered from the factors of $\widehat{\mathcal{M}}_3$ by using the inverse transformation matrix $Q^{\dagger} = Q^T I_{d \times K}$, and because $Q$ is perfectly conditioned, this inverse transformation does not introduce any errors due to numerical instability.

Instead of reducing the dimension to equal the number of topics, RPCA performs better with oversampling. The results of Corollary \ref{mainTop} require the tensor we decompose to have sufficiently large dimensions compared to the number of expected topics, so we can use the transformation matrix $Q' \defeq I_{K' \times d} U^T$, where $K'$ is chosen so that $K \le \rho_r \left( \frac{K^{'3} - m}{3 K' \log(K^{'3} - m) \alpha_0^4 \mu_0^2} \right)^{1/2}$, in accordance with Corollary \ref{mainTop}.

\section{Numerical Experiments}
\label{numerics}

We compare our model to existing methods for tensor and matrix RPCA on synthetic data and the escalator video dataset of \cite{LiData}. Our experiments demonstrate that our model significantly outperforms existing methods for tensor and matrix RPCA. We also see that we perform much better than the guarantees given in Theorem \ref{main}. Most remarkably, our model is able to recover tensors whose rank is much larger than its side lengths. In this regime, the Tucker-rank of the tensor is no longer an appropriate measure of the complexity of the data, and many existing methods for tensor RPCA \cite{GoldfarbTensors,SNN} are ineffective.
\begin{figure}[t]
\centering
\begin{tikzpicture}
\node (img1)  {\subfloat[\tiny This work]{\includegraphics[width=0.19\linewidth]{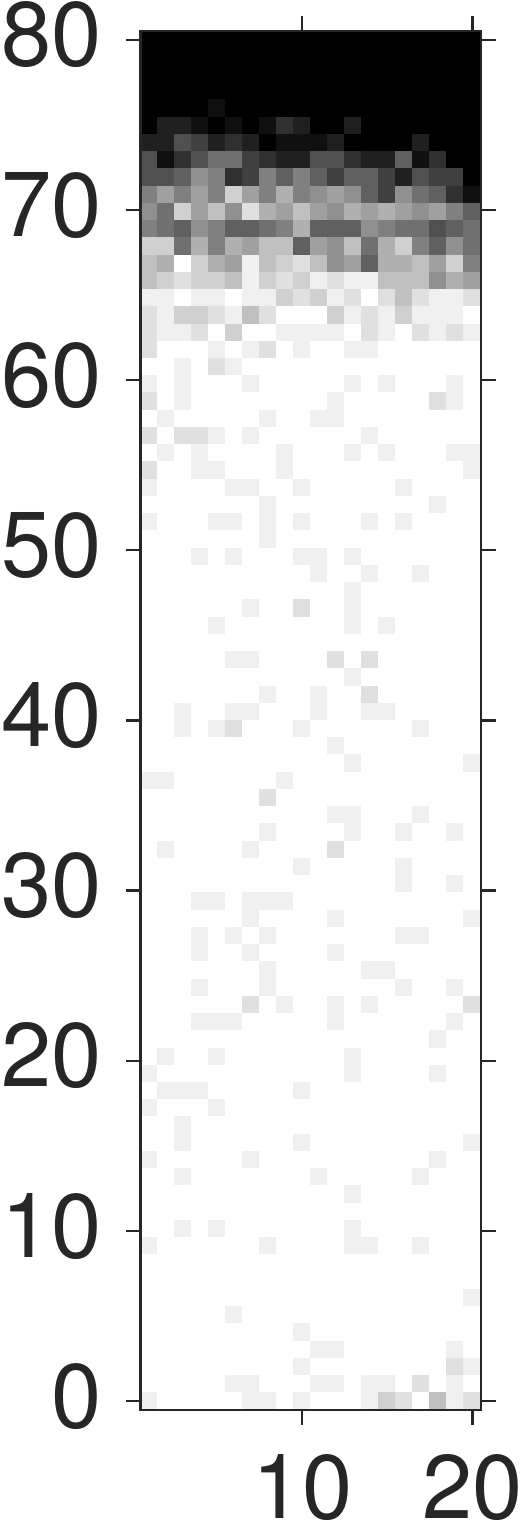}}
    \subfloat[\tiny \texttt{HoRPCA-S}]{\includegraphics[width=0.19\linewidth]{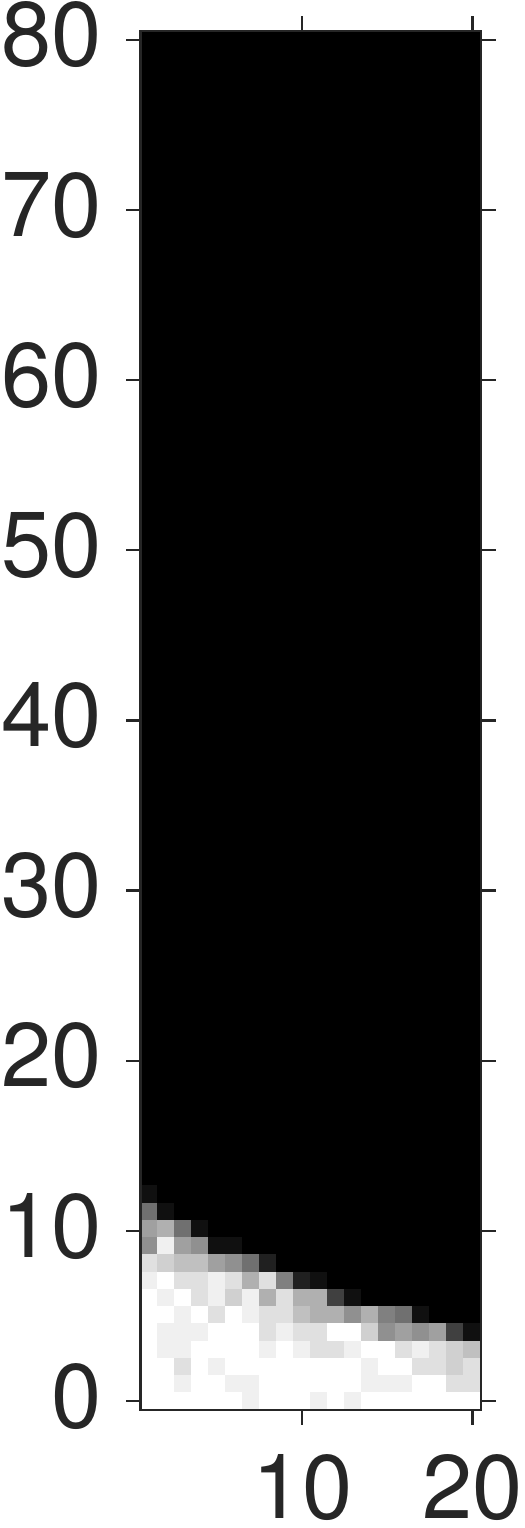}}
    \subfloat[\tiny \texttt{HoRPCA-C}]{\includegraphics[width=0.19\linewidth]{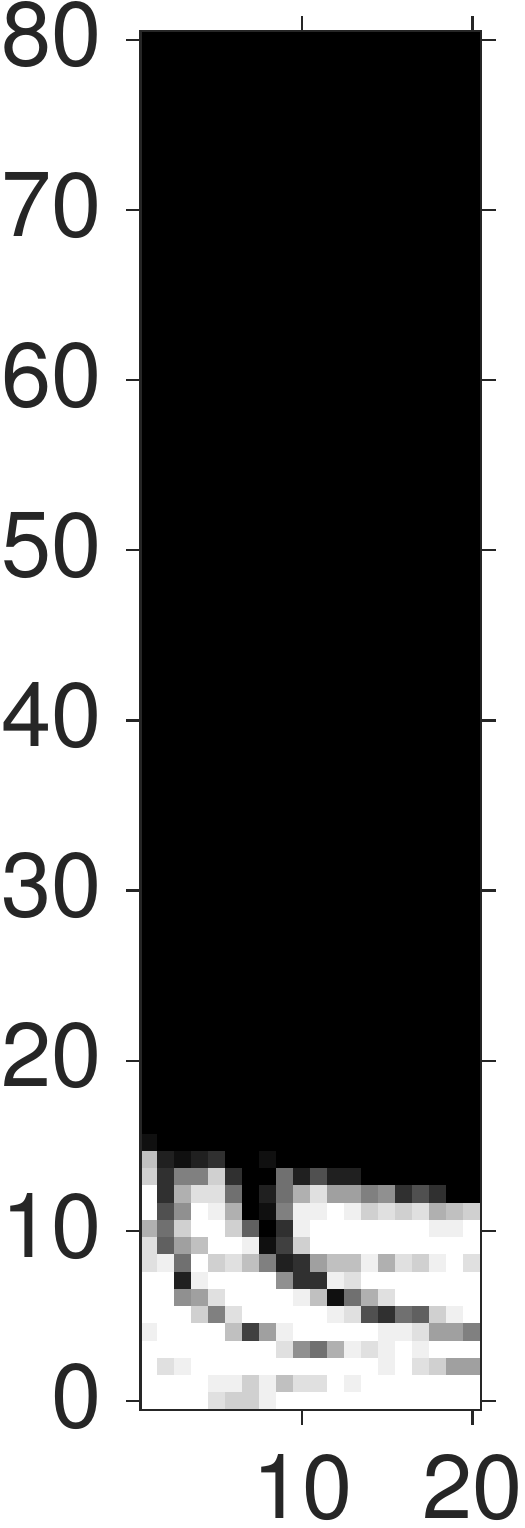}}
    \subfloat[\tiny Matrix RPCA]{\includegraphics[width=0.19\linewidth]{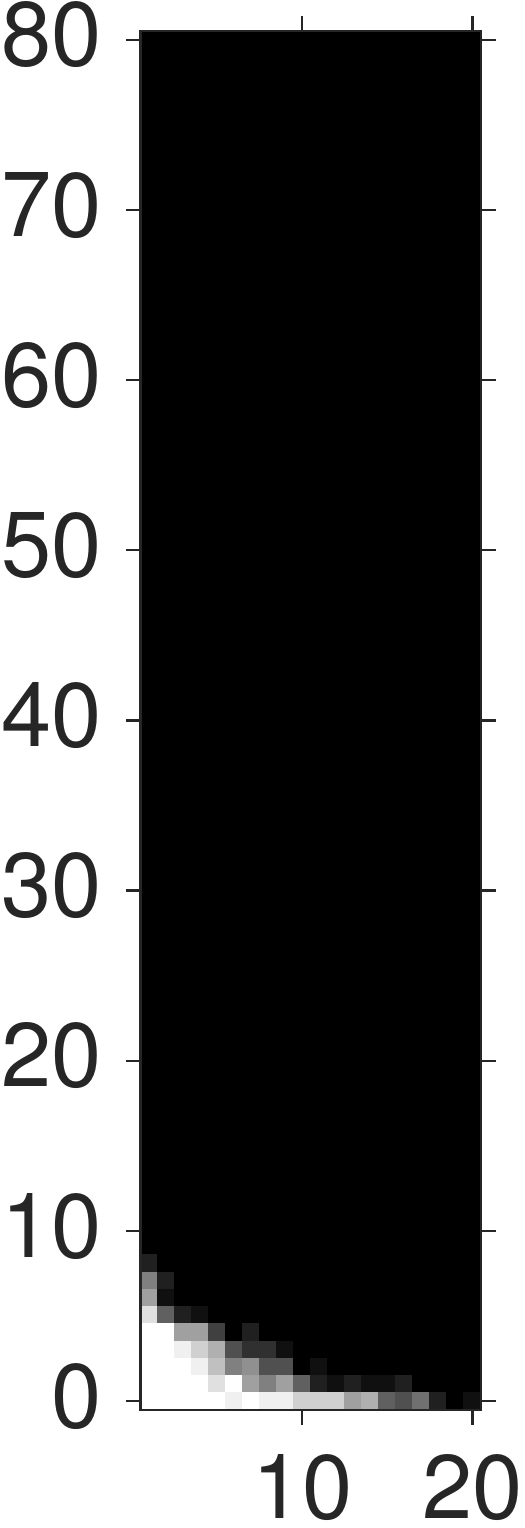}}};
\node[below=of img1, node distance=0cm, yshift=.9cm] {Percent sparsity of $\S$, $m/20^3$};
\node[left=of img1, node distance=0cm, rotate=90, anchor=center] {CP-rank of $\X$};
\end{tikzpicture}
\caption{A comparison of RPCA methods for recovering the decomposition $\Z = \X + \S$ with $\Z \in \R^{20 \times 20 \times 20}$. A pixel is colored white if $\X$ is recovered exactly. Each pixel represents the average of 16 trials. Algorithms (b), (c), and (d) are ill-posed when the CP-rank is greater than 20.}
\label{fig:synth}
\end{figure}

\subsection{Experiments on Synthetic Data}
\label{synth}

For each trial, we create an order-three, cubic dataset that can be represented as the sum of a low-rank tensor and a sparse tensor. To form the low-rank component, we randomly generate three factor matrices $A, B, C \in \R^{20 \times R}$, with each entry drawn i.i.d.\  $\mathcal{N}(0,1)$. We then form the low-rank component as
$
\X = \sum_{r=1}^R (a_r \otimes b_r \otimes c_r)
$.
We set the rank bound to be $R+10$ in our algorithm.

To form the sparse component, we make the tensor $\S \in \R^{20 \times 20 \times 20}$ with support chosen uniformly at random without replacement such that there are $m$ nonzeros, and  non-zero entries drawn i.i.d.\   $\mathcal{N}(0,1)$. Our ``observed'' dataset is then $\Z = \X + \S$. We vary both the rank $R$ and  sparsity $m$.

For each test, we perform 16 trials and measure the error between the recovered low-rank component $\X'$ and the actual low-rank component $\X$ using the relative least-squares loss $\frac{\|\X' - \X\|_F}{\|\X\|_F}$,
declaring ``exact recovery'' when this error is below $10^{-3}$. We fit our model using L-BFGS as implemented in \cite{SchmidtCode}, maintaining 10 iterations in memory, and we stop each trial after 1,000 iterations. For our parameters, we set $\lambdaX = 10^{-5}$ and $\lambdaS = 10^{-3}$. Both are small because we do not expect there to be any noise in $\Z$. Our results are shown in Figure \ref{fig:synth}, along with the results of the same experiment using matrix RPCA and two existing tensor RPCA methods.

The ranks reported in the figure are upper bounds, as it is possible for a certain $\X = \sum_{r=1}^R (a_r \otimes b_r \otimes c_r)$ to admit a lower-rank CP-decomposition. The Tucker-rank of the low-rank component is $(R,R,R)$ for $R \le 20$, and it is $(20,20,20)$ for $R \ge 20$. This is explicitly checked for each trial.

The two other tensor-based models, \texttt{HoRPCA-C} \cite{GoldfarbTensors} and \texttt{HoRPCA-S} \cite{GoldfarbTensors,SNN}, use the sum-of-nuclear-norms (SNN) regularizer, and \texttt{HoRPCA-S} is one of the only provable method for recovering low-rank tensors outside of this work \cite{SNN}. \texttt{HoRPCA-S} solves the problem
\begin{equation*}
\min_{\X, \S} \quad 
20 \sum_{i = 1}^K
\|\X'_{(i)}\|_* + \|\S'\|_{\textnormal{sum}}, \quad \textnormal{subject to:} \quad \X' + \S' = \Z,
\end{equation*}
See \eqref{eqSNN} for their recovery guarantees.
In our experiments, we see that \texttt{HoRPCA-S} performs worse than our tensor RPCA, and marginally better than matrix RPCA. When $R \ge 20$, \texttt{HoRPCA-S} is an ill-posed problem, because each matricization of $\X$ has full rank.

\texttt{HoRPCA-C} exhibits similar behavior. This program is defined as
\begin{equation}
\min_{\X', \S'} \quad \|\S'\|_{\textnormal{sum}} \quad \textnormal{subject to:} \quad \X' + \S' = \Z, \quad \rank(\X'_{(i)}) \le r_i.
\end{equation}
Although nonconvex, it has been shown to outperform other methods for tensor RPCA, including \texttt{HoRPCA-S} \cite{GoldfarbTensors}. However, this model still suffers from the effects of tensor matricization. For $R \ge 20$ in our experiments, \texttt{HoRPCA-C} is an ill-posed problem. For our implementation of \texttt{HoRPCA-C}, we set $r_i = R+1$ for all $i$. These rank bounds are much tighter than the rank bound $(R+10)$ we used for our model, but because $r_i$ cannot be larger than 20, it would not make sense to use the rank bound $R+10$ for the components of the Tucker-rank.

For the matrix RPCA in our experiments, we matricize the tensors and use the variational approach to RPCA developed in \cite{aravkin2014variational}.\footnote{Matrix RPCA code available at  \url{https://github.com/stephenbeckr/fastRPCA/}} The program we solve is
\begin{equation}
\min_{\X',\S'} \quad \max( \|\X'_{(1)}\|_*, \lambda \|\S'_{(1)}\|_{\textnormal{sum}} ) \quad \textnormal{s.t.} \quad \frac{1}{2} \|\X'_{(1)} + \S'_{(1)} - \Z_{(1)}\|^2_F \le \epsilon.
\end{equation}
One of the benefits of this approach is that we can choose $\lambda$ optimally because we know the matrices we would like to recover. We set $\lambda = \frac{\|\X_{(1)}\|_*}{\|S_{(1)}\|_1}$, and because we do not expect there to be any noise in $\Z$, we choose $\epsilon = 10^{-5}$.

\subsection{Tensor RPCA for Background Subtraction}

One of the most natural applications for RPCA is in background subtraction. In this section, we use our model to identify subjects in the ``escalator-video'' dataset provided by \cite{LiData}. This dataset is challenging for background-subtraction models because it contains three moving parts: a time-stamp, escalators, and the subjects. A strong model would be able to recognize that the motion of the escalators and the time-stamp is periodic, so these features belong to the low-rank component of the dataset, and the unpredictable motion of the subjects should be extracted into the sparse component.

\begin{figure}
\centering
\includegraphics[width=.95\linewidth]{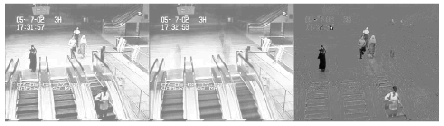}
\includegraphics[width=.95\linewidth]{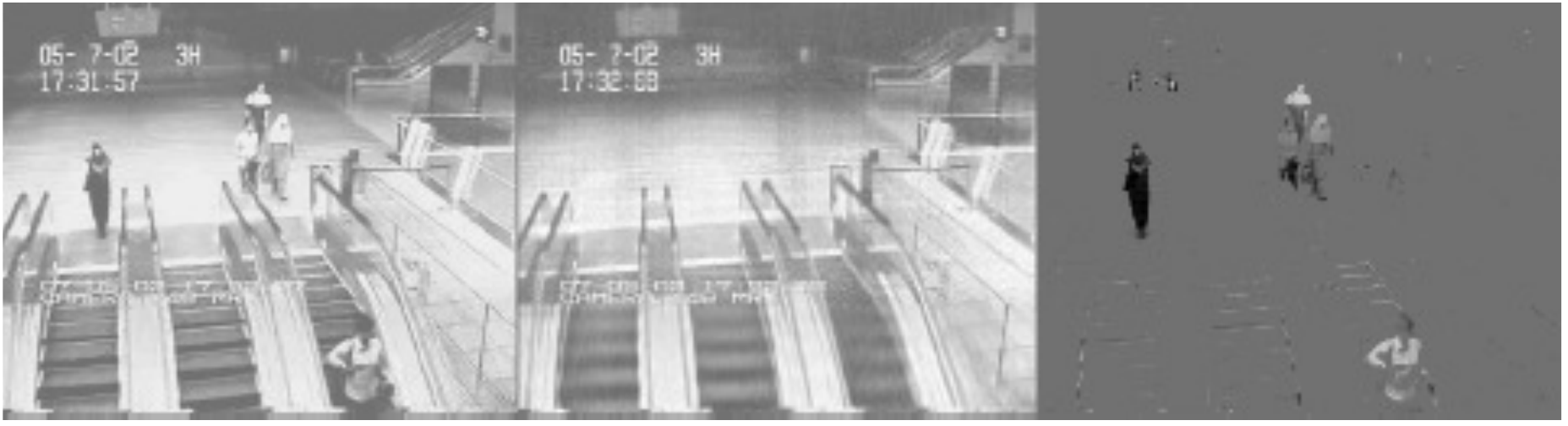}
\caption{\small The results of using (top) matrix RPCA and (bottom) tensor RPCA for background subtraction. From left to right: original image, low-rank component $(\X)$, and sparse component ($\S$). We see that tensor RPCA can more precisely extract the subjects into the sparse component and leave the moving escalator in the low-rank component. Frame 5 shown.}
\label{fig:movie}
\end{figure}

In Figure \ref{fig:movie}, we compare the performance of our tensor RPCA model and matrix RPCA for identifying the background in surveillance video. The video consists of 200 frames of size $130 \times 160$, which we store in a tensor of size $130 \times 160 \times 200$ or a matrix of size $20800 \times 200$. For tensor RPCA, we set the parameters $\lambdaX = 30, \ \lambdaS = 0.1$, and the rank-bound $R = 50$. For matrix RPCA, we choose $\lambda = 0.02$ and $\epsilon = 9 \times 10^{3}$. All of these parameters were chosen after careful tuning. The parameter $\epsilon$ was chosen to match the noise-level in the tensor RPCA solution to make the results more comparable. We see in Figure \ref{fig:movie} that tensor RPCA recovers a qualitatively superior decomposition, with the sparse component containing the subjects and very little of the stairs, and the low-rank component sufficiently ``sharp.'' In contrast, the low-rank component found by matrix RPCA appears more smoothed and has ``ghosts'' where the subjects should be removed, while the sparse component contains a significant amount of the stairs.

Even more impressive is the quantitative difference between the decompositions. The numerical rank of the low-rank component found using tensor RPCA is 48, and the sparse component is  5.5\%-sparse. For matrix RPCA, the recovered low-rank component has rank 58, and the sparse component is 49.3\%-sparse. It is clear that tensor-RPCA recovers a sparse component with significantly fewer non-zero entries than the sparse component that matrix RPCA recovers. To compare the low-rank components, it is better to compare degrees of freedom than the actual ranks. A tensor of size $d_1 \times d_2 \times d_3$ with CP-rank $r$ has $r (d_1 + d_2 + d_3)$ degrees of freedom, so tensor RPCA finds a low-rank component with \num{23520} degrees of freedom, and matrix RPCA finds a low-rank component with \num{1218000} degrees of freedom, which is an enormous difference in the complexity of the solutions.

\subsection{Topic Modeling Experiments}
\label{sec:numExTop}

Using the \texttt{20newsgroups} dataset \cite{Lang95}, 
we compare our tensor RPCA strategy to the tensor approach of \cite{AnandHsu} and the standard variational Bayes approach.
We use a training set of 11,314 documents, and compare our results on a testing set of 7,532 documents. Our vocabulary consists of the 1,000 most common words in the training-set corpus, disregarding words that occur in more that 95\% of the documents and those that occur in fewer than five documents. We also removed all English stopwords given by the University of Glasgow's stopword list. We look for five topics to discover, hoping to discover broad topics related to technology, recreation, science, politics, and religion. We do not include headers, footers, or quotes.

For the variational Bayes tests, we use SciKit-Learn's implementation of Latent Dirichlet Allocation, with hyperparameters $\alpha = 0.1, \beta = 0.1, \kappa = 0.7$, and $\tau_0 = 50$. We run the solver for twenty iterations. We use code provided by the authors of \cite{Anand1} for our comparison to their method. For these tests, we set $\alpha_0 = 0.01$ and the number of topics $K=5$. For tensor RPCA, we oversample, setting $K' = 10$, and we choose the other parameters to $\lambda = 10^{-3}$ and $\mu = 10^{-8}$. The topics that each approach finds are included in Appendix \ref{app:exper}.

The decomposition from tensor RPCA has numerical rank five, showing that it discovers five topics in the data. It finds 87 sparse errors. We quantitatively compared our results using the perplexity on the test set, defined for each word $w$ in the test set as \[ \textnormal{perplexity}(w) = \exp\left(-\frac{\mathcal{L}(w)}{d} \right),\]
where $d$ is the size of the vocabulary and $\mathcal{L}(w) = \log p(w|\Phi,\alpha)$ is the log-likelihood of the test set given the estimated topic-matrix $\Phi$ and hyperparameter $\alpha$. To quantitatively measure the interpretability of the topics, we used the observed coherence measure from \cite{teaLeaves}. A comparison using these measures is shown in Figure \ref{tab:quant}. Qualitatively, it appears that variational Bayes has the best topic coherence, but it also finds meaningless topics, grouping low-frequency words or numbers together, which are not article topics. The tensor-based approaches do not exhibit this problem.

\begin{figure}
\centering
\begin{tabular}{l c c}
\toprule
& Perplexity & Average Observed Coherence \\
\midrule
Variational Bayes & 1417 & 0.033 \\
Tensor CPD & 1503 & 0.006 \\
Tensor RPCA & 1309 & 0.007 \\
\bottomrule
\end{tabular}
\caption{A comparison of topic models on the 20newsgroups dataset. We use SciKit-Learn's implementation of variation Bayes for LDA and the method of \cite{Anand1} for the tensor CPD. The observed coherence for each topic is calculated using the method of \cite{teaLeaves}, and we present the average observed coherence over all five topics.}
\label{tab:quant}
\end{figure}

The tensor decomposition step was a bit faster using the approach of \cite{AnandHsu}, taking 3.46 seconds for the decomposition step to converge compared to tensor RPCA's 5.13 seconds. We measure convergence with respect to the relative change in the low-rank component, $\frac{\|\X_{k} - \X_{k-1}\|_F}{\|\X_k\|_F}$. We set the tolerance to $10^{-4}$ for both approaches. Both tensor-based methods are significantly faster than variational Bayes, which took 44.71 seconds to perform twenty iterations.

\section{Conclusion}

Our guarantees show that tensor RPCA with atomic-norm regularization outperforms matrix-based RPCA and RPCA algorithms based on matricization. Although the atomic norm is generally intractable, our use of a higher-order generalization of Burer-Monteiro factorization allows us to derive a nonconvex program equivalent to tensor RPCA. Our nonconvex model can be fit efficiently using any first-order optimizer. While convergence to a global optimum is not generally ensured, we provide sufficient conditions for a local minimum to be globally optimal which can be verified \emph{ex post facto}.

As an algorithm for the low-rank decomposition of tensors, tensor RPCA can be used for training many latent variable models, including LDA. In this context, our main result offers performance guarantees for estimating relevant parameters. Tensor RPCA offers many improvements over existing methods for decomposing moment-tensors. Notably, it does not require the underlying distributions to be linearly independent, and it avoids the unstable whitening step required by many existing methods. Tensor RPCA also provably recovers the population moment-tensor in the presence of sparsely distributed errors in the sample moments.

Empirically, our tensor RPCA significantly outperforms matrix RPCA as well as existing implementations of RPCA that use sum-of-nuclear-norm regularization. Our approach to tensor RPCA is also able to recover tensors whose CP-rank greatly exceeds all of its side lengths, a regime where sum-of-nuclear-norm models are ill-posed. Our results suggest that analyzing low-rank tensor recovery in terms of the Tucker-rank does not yield tight performance bounds, and future work might investigate performance guarantees in terms of the CP- or atomic-rank.

\bibliographystyle{abbrv}
\bibliography{thesisDriggs.bib}

\clearpage

\begin{appendices}

\section{Topic Model Results}
\label{app:exper}

\begin{table}[h]
\begin{tabular}{l p{5in} }
  \toprule  
 Topic & Associated words \\
 \midrule
1& 10 00 space 20 new 15 25 12 11 14 gun 30  000 16 1993 17 18 year 13 president april 50 mr national states \\
\midrule
2 & edu file use windows com program thanks drive like available software mail does using files version data know card problem window ftp email info used \\
\midrule
3 & key use government chip public used encryption keys like new law security clipper  information bike bit privacy using number data don private does technology make \\
\midrule
4 & people don just think like know time god good say did said does  way right ve believe make going really things want didn ll years \\
\midrule
5 & ax max b8f g9v a86 pl 145 1d9 0t 34u 1t 3t 
giz bhj wm 2di 75u 2tm cx bxn 7ey w7 chz sl 0d 
\\
  \bottomrule  
\end{tabular}
\caption{Topics found using variational Bayes for LDA parameter estimation}
\end{table}

\begin{table}[h]
\begin{tabular}{l p{5in} }
  \toprule  
 Topic & Associated words \\
 \midrule
1 & effect sorry articles come unix current main trying dos doesn involved encryption likely 
report day news sense administration files interesting lots mode suggest sources key \\
\midrule
2 & dos perfect key thinking goes knowledge likely doesn saying didn rights ways best 
sale times reason think jewish got makes wanted good switch die things \\
\midrule
3 & likely news key driver times got used version thinking knowledge dos gov theory 
makes problems useful years effect san needed ways local 100 card windows \\
\midrule
4 & knowledge doesn theory dos likely key wire apparently main used needed age hi 
information high looks administration cards approach function programs input wanted meaning encryption \\
\midrule
5 & b8f windows used maybe card version wire programs come israeli graphics best dos 
rights age problem articles difficult manual keyboard distribution usually happen little switch \\
\bottomrule
\end{tabular}
\caption{Topics found using the tensor-based method of moments from \cite{Anand1}. ALS was used for the tensor decomposition.}
\end{table}

\begin{table}
\begin{tabular}{l p{5in} }
  \toprule  
 Topic & Associated words \\
 \midrule
 1 & b8f news 000 card likely years games got teams times yes san dos 
reason 100 subject key gas thinking second didn agree things ways rights \\
\midrule
2 & likely key knowledge dos effect doesn thinking perfect used got theory news times 
version needed ways wanted rights reason makes useful didn main saying wire \\
\midrule
3 & effect sorry articles unix current encryption news san 000 involved conference come day 
got report suggest assume sense 100 1d9 department united main printer office \\
\midrule
4 & b8f theory doesn dos card maybe age driver games gas used files wire 
lord high function main source problems approach term looks programs letter apparently \\
\midrule
5 & theory effect knowledge doesn main wire age encryption high files looks hi information 
apparently administration function programs input cards final needed approach used solution sense \\
\bottomrule
\end{tabular}
\caption{Topics found using the tensor-RPCA based method-of-moments in this work}
\end{table}

\FloatBarrier

\section{Equivalence of Bernoulli and Uniform Sampling Models}
\label{app:berunif}

We would like to show that if conditions \eqref{condWS}(d) and (e) hold with high probability when the support of $\S$ follows a Bernoulli distribution with parameter $2 \rho$, then the same conditions hold with high probability when the support of $\S$ is uniformly distributed over all sets of cardinality $m$. We also need to show that the conclusions of Lemma \ref{lem:opBound} hold under the Bernoulli model as well as the uniform model. We first prove equivalence between the Bernoulli and uniform models for any event whose probability decreases with $|\Omega|$, and then we show that conditions \eqref{condWS}(d) and (e) and Lemma \ref{lem:opBound} satisfy this property.

\begin{lemma}
\label{lem:bernUnif}
Let $E$ be any event whose probability decreases as $|\Omega|$ increases. If $E$ holds with high probability when $\Omega$ follows the Bernoulli distribution of \eqref{eq:berdist}, then $E$ holds with high probability when the support of $\Omega$ is uniformly distributed over all sets of cardinality $m$, and the converse holds as well.
\end{lemma}

\begin{proof}
Our proof is similar to \cite[App. 7]{RPCA}. Let $\mathbb{P}_{\textnormal{Ber}(\rho)}$ and $\mathbb{P}_{\textnormal{Unif}(k)}$ denote the probabilities calculated under the Bernoulli and uniform models, respectively. We begin by showing that if $E$ holds with high probability under the Bernoulli model, then it also holds with high probability under the uniform model.
\begin{align}
    \mathbb{P}_{\textnormal{Ber}(\rho)} \left( E \right) & = \sum_{k=0}^{d_1 d_2 d_3} \mathbb{P}_{\textnormal{Ber}(\rho)} \left( E \big| |\Omega| = k \right) \mathbb{P}_{\textnormal{Ber}(\rho)} \left( |\Omega| = k \right) \\
    & \le \sum_{k=0}^{m-1} \mathbb{P}_{\textnormal{Ber}(\rho)} \left( |\Omega| = k \right) + \sum_{k=m}^{d_1 d_2 d_3} \mathbb{P}_{\textnormal{Unif}(k)} \left(E\right) \mathbb{P}_{\textnormal{Ber}(\rho)} \left( |\Omega| = k \right) \\
    &\le \mathbb{P}_{\textnormal{Ber}(\rho)}\left( |\Omega| < m \right) + \mathbb{P}_{\textnormal{Unif}(m)}(E).
\end{align}
Here, we have used the facts that the distribution of $\Omega$ conditioned on its cardinality is uniform, and that $\mathbb{P}_{\textnormal{Unif}(k)} \left( E \right) \le \mathbb{P}_{\textnormal{Unif}(m)} \left( E \right)$ for all $k \ge m$. This implies that
\begin{equation}
    \mathbb{P}_{\textnormal{Unif}(m)} \left( E \right) \ge \mathbb{P}_{\textnormal{Ber}(\rho)} \left( E \right) - \mathbb{P}_{\textnormal{Ber}(\rho)} \left( |\Omega| < m \right).
\end{equation}
Choosing $\rho = \frac{m}{d_1 d_2 d_3} + \epsilon$ for $\epsilon > 0$, we have $\mathbb{P}_{\textnormal{Ber}(\rho)} (|\Omega| < m) \le e^{-\frac{\epsilon^2 d_1 d_2 d_3}{2 \rho}}$. This proves that $E$ holds with high probability under the uniform model.

The converse holds as well. Suppose $E$ holds under the uniform model. Then
\begin{align}
    \mathbb{P}_{\textnormal{Ber}(\rho)} \left( E \right) & \ge \sum_{k=0}^{m} \mathbb{P}_{\textnormal{Ber}(\rho)} \left( E \big| |\Omega| = k \right) \\
    & \ge \mathbb{P}_{\textnormal{Unif}(m)}(E) \sum_{k=0}^m \mathbb{P}_{\textnormal{Ber}(\rho)}(|\Omega| = k) \\
    &= \mathbb{P}_{\textnormal{Unif}(m)}(E) \mathbb{P}_{\textnormal{Ber}(\rho)}(|\Omega| \le m).
\end{align}
Choosing $m$ large enough ensures that $\mathbb{P}(|\Omega| > m)$ is exponentially small, so $E$ holds with high probability under the Bernoulli model as well.
\end{proof}

All that is left is to show that conditions \eqref{condWS}(d) and (e) and Lemma \ref{lem:opBound} satisfy the assumptions of the previous lemma.

\begin{lemma}
        Let $\S_0$ have support set $\Omega_0$, and let $\S_0'$ have support set $\Omega_0' \subsetneq \Omega_0$. Let $\W^{S_0} = \lambda \PP_{\X^{\perp}} (\PP_{\Omega_0} - \PP_{\Omega_0} \PP_{\X} \PP_{\Omega_0})^{-1} \sgn(S_0)$ and $G_0 = \sgn(\S_0)$, with $S_0'$ and $G_0'$ defined analogously. Then
        \begin{equation}
        \label{eq:berunif0}
            \mathbb{P} \left( \PP_{\X} (\tfrac{d_1 d_2 d_3}{n} \PP_{\Omega_0} - I)\PP_{\X} \le \sigma \right) \le \mathbb{P} \left( \PP_{\X} (\tfrac{d_1 d_2 d_3}{n} \PP_{\Omega_0'} - I)\PP_{\X} \le \sigma \right),
        \end{equation}
        \begin{equation}
        \label{eq:berunif1}
            \mathbb{P} \left( \| \W^{\S_0} \| < \frac{1}{8} \right) \le \mathbb{P} \left( \| \W^{\S_0'} \| < \frac{1}{8} \right),
        \end{equation}
        and
        \begin{equation}
        \label{eq:berunif2}
             \mathbb{P} \left(  \|\PP_{\Omega_0^{\perp}} \W^{\S_0} \|_{\max} < \frac{\lambda}{8} \right) \le \mathbb{P} \left(  \|\PP_{\Omega_0^{' \perp}} \W^{\S_0'} \|_{\max} < \frac{\lambda}{8} \right).
        \end{equation}
\end{lemma}

\begin{proof}
Inequality \eqref{eq:berunif0} is clear from Lemma \ref{lem:opBound}, and this inequality implies the other two. Consider \eqref{eq:berunif1}. The bound in \eqref{eq:12d} shows that
\begin{align}
    \mathbb{P}\left( \lambda \left\| \sum_{k=1}^\infty (\PP_{\Omega_0} \PP_{\X} \PP_{\Omega_0})^k G_0 \right\| > t \right) &\le 2 \exp\left( -\frac{t^2 (1-\sigma^2)^2}{32 \sigma^4 \lambda^2} + \frac{21 (d_1+d_2+d_3)}{4} \right) \\
    & \quad + \mathbb{P}( \| \PP_{\Omega_0} \PP_{\X} \| > \sigma ).
\end{align}
With $\sigma$ fixed, the only dependence on $\Omega_0$ is in the second term on the right. It is clear from \eqref{eq:berunif0} that this probability is decreasing $|\Omega_0|$. Because $|\Omega_0'| \le |\Omega_0|$ this implies \eqref{eq:berunif1}. The last inequality, \eqref{eq:berunif2}, follows from the bound in \eqref{eq:Firste} by the same argument.

\end{proof}

Finally, we need to show that exact recovery under the random-sign model implies exact recovery under the fixed sign model.
\begin{lemma}
\label{lem:deran}
    Suppose $\LC$ satisfies the conditions of Theorem \ref{main} and that the locations of the nonzero entries of $\S$ follow the Bernoulli model with parameter $2\rho$. Assume further that the signs of $\S$ follow a Bernoulli model with parameter $\rho$. If the solution to \eqref{RPCA} is exact with high probability, then it is also exact with at least the same probability for the model in which the signs are fixed and the locations are sampled from the Bernoulli model with parameter $\rho$.
\end{lemma}
\begin{proof}
Lemma \ref{lem:deran} is proved for the matrix case in \cite[Thm. 2.3]{RPCA}, and generalizing this proof to higher-orders is trivial.
\end{proof}

\section{Proof of Lemma \ref{lem:SubGaussBound}}
\label{app:tomi}

The following lemma is implied by Lemma 1 and Theorem 1 of \cite{Tomioka}.

\begin{lemma}
    Suppose that each element of $\X_{i_1,i_2,i_3}$ is independent, zero-mean, and satisfies $\mathbb{E}\left[ e^{t\X_{i_1,i_2,i_3}} \right] \le e^{k^2 t^2/2}$. Then the spectral norm of $\X$ can be bounded as follows:
    \begin{equation}
       \|\X\| \le \sqrt{8 k^2 \left( \sum_{i=1}^3 d_i \right) \log(6/\log(3/2)) + \log(2/\delta) }
    \end{equation}
    with probability at least $1 - \delta$.
\end{lemma}
To prove Lemma \ref{lem:SubGaussBound}, we just need to show that the distribution of $G$ is sub-Gaussian:
\begin{align}
    \mathbb{E}\left[ e^{t\X_{i_1,i_2,i_3}} \right] = \rho_s \left( \frac{e^{t^2} + e^{-t^2}}{2} \right) \le \cosh(t^2) \le e^{t^2/2}.
\end{align}

    \section{Proof of Proposition \ref{prog2}}
    \label{sec:proofBurerMonteiro}
    
    Our proof of Proposition \ref{prog2} uses the following proposition:
    \begin{proposition}
        \label{prog1}
        Suppose $\gamma^*, \{ u_r^{* (1)} \}, \cdots, \{ u_r^{*(K)} \}, \S^*$ are optimal for the nonconvex program
        \begin{align}
        \label{noncon3}
        \min_{\gamma, \{ u_r^{(1)} \}, \cdots, \{ u_r^{(K)} \}} &\frac{1}{2} \left\| \sum_{r=1}^R \gamma_r (u_r^{(1)} \otimes \cdots \otimes u_r^{(K)} ) + \S - \Z \right\|_F^2 + \lambdaS \| \S \|_{\textnormal{sum}} \\
        & \quad \quad \quad \quad \quad \quad \quad \quad \quad \quad \quad \quad \quad \quad \quad \quad \quad + \frac{\lambdaX}{K} \|\gamma\|_1 \\
        \textnormal{s.t.} \quad & \|u_r^{(1)}\|, \cdots, \|u_r^{(K)}\| \le 1,
        \end{align}
        Let $\X^* = \sum_{r=1}^R \gamma^*_r (u^*_r \otimes v^*_r \otimes w^*_r)$. Then the point $(\X^*,\S^*)$ is optimal for the problem \eqref{RPCAlag}. Conversely, if $(\X^*,\S^*)$ is the minimizer of \eqref{RPCAlag}, then the terms $\gamma^*, \{ u_r^* \}, \{ v_r^* \}, \{ w_r^* \}$ from a decomposition of $\X^*$ are optimal for \eqref{noncon3}.
    \end{proposition}
    \begin{proof}
        Using the definition of the atomic norm, we can rewrite \eqref{RPCA} as
        \begin{align}
        \min_{\X} \quad &\frac{1}{2} \left\| \X + \S - \Z \right\|_F^2 + \lambdaS \| \S \|_{\textnormal{sum}} + \min_{\gamma, \{ u_r^{(1)} \}, \cdots, \{ u_r^{(K)} \}} \lambdaX \|\gamma\|_1, \\
        \textnormal{s.t.} \quad & \X = \sum_{r=1}^R \gamma_r (u_r^{(1)} \otimes \cdots \otimes u_r^{(K)} ), \quad \|u_r^{(1)}\| = \cdots = \|u_r^{(K)}\| = 1.
        \end{align}
        Due to the coerciveness of norms, replacing the norm constraints with inequalities does not change the global optima. This adjustment yields \eqref{noncon3}.
    \end{proof}
    \noindent If $R$ is chosen large enough, then the program \eqref{noncon3} is an equivalent reformulation of \eqref{RPCA}. However, instead of replacing the atomic norm with a smooth, nonconvex regularizer as we would in the matrix case, we have introduced a non-smooth term and multiple constraints. We would like a nonconvex representation of the atomic norm that more closely generalizes \eqref{nonconnuc}. Proposition \ref{prog2} provides this. The proof of Proposition \ref{prog2} is as follows:

    \begin{proof}
        We use an argument similar to the proof in Appendix II of \cite{Bazerque}. We can rewrite \eqref{noncon1} as
        \begin{align}
        \label{noncon2}
        \min_{\gamma, \{ a_r^{(1)} \}, \cdots , \{a_r^{(K)}\}, \{u_r^{(1)} \}, \cdots , \{u_r^{(K)}\} } \ &\frac{1}{2} \left\| \X + \S - \Z \right\|_F^2 + \lambdaS \| \S \|_{\textnormal{sum}} \\
        & \quad \quad \quad \quad \quad \quad \quad \ + \frac{\lambdaX}{K} \sum_{r=1}^R \sum_{k = 1}^K \|a_r^{(k)}\|^K \\
        \textnormal{s.t.} \quad & \X = \sum_{r=1}^R \gamma_r (u_r^{(1)} \otimes \cdots \otimes u_r^{(K)}), \notag\\
        & \gamma_r = \| a_r^{(1)} \| \cdots \|a_r^{(K)}\|.\notag
        \end{align}
        Minimizing over $\gamma, \{a_r\}, \{b_r\},$ and $\{c_r\}$ first, we must solve
        \begin{align*}
        \min_{\gamma} \quad &\sum_{r=1}^R \sum_{k=1}^K \|a_r^{(k)}\|^K \\
        \textnormal{s.t.} \quad &\gamma_r = \| a_r^{(1)} \| \cdots \|a_r^{(K)}\|.
        \end{align*}
        The AM-GM inequality tells us
        \begin{equation*}
        ( \| a_r^{(1)} \|^K \cdots \|a_r^{(K)}\|^K )^{\frac{1}{K}} \le \tfrac{1}{K} (\| a_r^{(1)} \|^K + \cdots + \|a_r^{(K)}\|^K),
        \end{equation*}
        with equality when $\|a_r^{(1)}\| = \cdots = \|a_r^{(K)}\| = \gamma^{\frac{1}{K}}$, so the optimal $\gamma$ satisfies
        \begin{equation*}
        \|\gamma\|_1 = \frac{1}{K} \sum_{r=1}^R \sum_{k=1}^K \| a_r^{(k)} \|^K.
        \end{equation*}
        Using these optimal values in \eqref{noncon2}, we see that \eqref{noncon2} is equivalent to
        \begin{align*}
        \min_{\gamma, \{ u_r^{(1)} \}, \cdots, \{ u_r^{(K)} \}} \quad &\frac{1}{2} \left\| \sum_{r=1}^R \gamma_r (u_r^{(1)} \otimes \cdots \otimes u_r^{(K)} ) + \S - \Z \right\|_F^2 + \lambdaS \| \S \|_{\textnormal{sum}} \\
        & \quad \quad \quad \quad \quad \quad \quad \quad \quad \quad \quad \quad \quad \quad \quad \quad \quad \ + \frac{\lambdaX}{K} \|\gamma\|_1, \\
        \textnormal{s.t.} \quad & \|u_r^{(1)}\|, \cdots, \|u_r^{(K)}\| \le 1,
        \end{align*}
        which is equivalent to \eqref{RPCA} by Proposition \ref{prog1}.
    \end{proof}

\end{appendices}

\end{document}